\documentclass[review,10pt]{elsarticle}
\usepackage{lipsum}
\usepackage{amsthm}
\makeatletter
\def\ps@pprintTitle{%
 \let\@oddhead\@empty
 \let\@evenhead\@empty
 \def\@oddfoot{}%
 \let\@evenfoot\@oddfoot}
\makeatother
\usepackage[margin=2.5cm]{geometry}
\usepackage{lineno,hyperref}
\usepackage{amsmath,amssymb,latexsym, amscd}
\usepackage{exscale, eps fig, graphics}
\usepackage{array}
\usepackage[usenames,dvipsnames]{color}
\usepackage{float}
\restylefloat{table}
\newcommand\numberthis{\addtocounter{equation}{1}\tag{\theequation}}
\usepackage{algorithmic}
\usepackage{placeins}
\usepackage{longtable}
\usepackage{makecell}
\usepackage[utf8]{inputenc}
\usepackage[justification=centering]{caption}
\usepackage{graphicx}
\usepackage{caption}
\usepackage{subcaption}
\usepackage{mathtools}
\usepackage{eucal}
\usepackage{mathrsfs}
\usepackage{fontenc}
\hypersetup{pdfauthor={Name}}
\newtheorem{thm}{Theorem}[section]
\newtheorem{lem}[thm]{Lemma}
\newtheorem{prop}[thm]{Proposition}
\newtheorem{rem}[thm]{Remark}

\newtheorem{cor}[thm]{Corollary}

\newcommand{\BigO}[1]{\ensuremath{\operatorname{O}\left(#1\right)}}

\renewcommand{\geq}{\geqslant}
\renewcommand{\leq}{\leqslant}

\newcommand{\BigOH}[1]{\ensuremath{\operatorname{O}_{H}\left(#1\right)}}
\newcommand{\BigOP}[1]{\ensuremath{\operatorname{O}_{P}\left(#1\right)}}
\newcommand{\BigOqP}[1]{\ensuremath{\operatorname{O}_{q}\left(#1\right)}}

\renewcommand{\P}{\boldsymbol{\phi}}
\renewcommand{\a}{\boldsymbol{\alpha}}

\def\mathbi#1{\textbf{\em #1}}

\begin{document}

\begin{frontmatter}

\title{On the distribution of polynomial Farey points and Chebyshev's bias phenomenon}

\author[mymainaddress]{Bittu Chahal}
\ead{bittui@iiitd.ac.in}

\author[mymainaddress]{Sneha Chaubey}
\ead{sneha@iiitd.ac.in}

\address[mymainaddress]{Department of Mathematics, IIIT Delhi, New Delhi 110020}

\begin{abstract}

We study two types of problems for polynomial Farey fractions. For a positive integer $Q$, and polynomial $P(x)\in\mathbb{Z}[X]$ with $P(0)=0$, we define polynomial Farey fractions as \[\mathcal{F}_{Q,P}:=\left\{\frac{a}{q}: 1\leq a\leq q\leq Q,\ \gcd (P(a),q)=1\right\}.\]
The classical Farey fractions are obtained by considering $P(x)=x$. In this article, we determine the global and local distribution of the sequence of polynomial Farey fractions via discrepancy and pair correlation measure, respectively. In particular, we establish that the sequence of polynomial Farey fractions is uniformly distributed modulo one and show that the limit superior of the pair correlation measure of $(\mathcal{F}_{Q,P})_{Q\ge1}$ is bounded. For the specific polynomial $P(x)=x(x+1)$, we show the existence of the limiting pair correlation measure of $(\mathcal{F}_{Q,P})_{Q\ge1}$ and also provide an explicit formula for the pair correlation function which is non-Poissonian. Further, restricting the polynomial Farey denominators to certain subsets of primes, we explicitly find the pair correlation measure and show it to be Poissonian. Finally, we study Chebyshev's bias type of questions for the classical and polynomial Farey denominators along arithmetic progressions and obtain an $\Omega$-result for the error term of its counting function.  

\end{abstract}


\end{frontmatter}

\section{Introduction and main results}

In this article, we focus on the global and local distribution of sequences emerging from the Farey subdivision of the unit interval. 
 Let $\mathbi{c}=(c_\nu,c_{\nu-1},\ldots,c_1)\in\mathbb{Z}^{\nu}$ be a fixed non-zero vector and $P(x)=c_{\nu}x^{\nu}+c_{\nu-1}x^{\nu-1}+\cdots+c_1x$ be a polynomial. Denote 
 \[\mathcal{F}_{Q,P}:=\left\{\frac{a}{q}: 1\leq a\leq q\leq Q,\ \gcd (P(a),q)=1\right\}.\numberthis\label{F_Q}\]If $P(x)=x(x+1)$ then for instance,
\[\mathcal{F}_{5,P}=\left\{\frac{1}{5}, \frac{1}{3}, \frac{2}{5},\frac{3}{5}, 1 \right\}. \]
The motivation to define this sequence comes from a geometric point of view via visibility. The classical Farey fractions of order $Q$ are given by
\[\mathcal{F}_Q=\left\{\frac{a}{q}: 1\le a\le q\le Q,\ \gcd(a,q)=1\right\}.\]
They are in a one-to-one correspondence with visible lattice points in the triangle with vertices $(0,0), (0,Q)$, and $(Q,Q)$ along straight lines passing through the origin. A point $(a,b)\in\mathbb{Z}^2$ is called visible from the origin if there is no other point of $\mathbb{Z}^2$ on the straight line joining the origin and point $(a,b)$. Hence, the sequence $\mathcal{F}_Q$ can be viewed in terms of visible lattice points: 
\[\mathcal{F}_Q=\{a/q\ |\ 1\leq a\leq q\leq Q;\ (a,q)\ \text{is visible from origin }\}.\]
Several mathematicians \cite{Roger, Cobeli, Nowak} have studied various problems on the distribution of visible lattice points in planar and convex domains and their generalizations. Recently, the authors in \cite{Chaubey} introduced polynomial visible lattice points. It is defined as: let $\mathbi{c}=(c_n,c_{n-1},\ldots,c_1)\in\mathbb{Z}^n$ be a fixed vector with $c_n\ne 0,\ c_i\geq 0$ for all $1\leq i\leq n$, and $\gcd(c_n,c_{n-1},\ldots,c_1)=1$, let $F(\mathbi{c})=\{y=rP_{\mathbi{c}}(x)\ |\ r\in\mathbb{Q}^+\}$, where $P_{\mathbi{c}}(x)=c_nx^n+c_{n-1}x^{n-1}+\cdots+c_1x$. A point $(a,b)\in\mathbb{N}^2$ is called $F(\mathbi{c})$-visible if there is no other lattice point on the curve $y=rP_{\mathbi{c}}(x)$ joining the origin and the point $(a,b)$. Denote the set of all $F(\mathbi{c})$-visible points in $\mathbb{N}^2$ by $V(\mathbi{c})$. It is proved in \cite{Chaubey} that the natural density of the polynomial visible lattice points is equal to one. Note that the natural density of the polynomial visible lattice points for the curve $y=mx^b$ for $b\ge 1$ is equal to $\dfrac{1}{\zeta(b+1)}$ \cite{MR3836421}. We denote the set of fractions $a/q\in\mathcal{F}_Q$ such that $(a,q)$ is $F(\mathbi{c})$-visible as:
\[\mathscr{F}_{Q,V}:=\left\{\frac{a}{q}\ |\ 1\leq a\leq q\leq Q,\ \gcd(a,q)=1,\ \text{and}\ (a,q)\in V(\mathbi{c}) \right\}. \]
This gives a relation between $\mathscr{F}_{Q,V}$ and the visible lattice points through polynomial curves in the triangle with vertices $(0,0),\ (0,Q)$, and $(Q,Q)$. It is natural to study the distributional properties of the above sequence but it turns out to be a difficult problem with no arithmetic description of the visibility property in this setting.  Instead, we focus on a subset of $\mathcal{F}_{Q,V}$ defined in \eqref{F_Q}  as
\[\mathcal{F}_{Q,P}=\left\{\frac{a}{q} \ | \ 1\leq a\leq q\leq Q,\ \gcd (P(a),q)=1\right\},\]
which is a subset of fractions that are in one-to-one correspondence with visible lattice points through polynomial curves. To see that $\mathcal{F}_{Q,P}\subsetneq\mathcal{F}_{Q,V}$, take $\frac{a}{q}\in\mathcal{F}_{Q,P}$
and choose $r=\frac{q}{P(a)}\in\mathbb{Q}^+$, so that $q=rP(a)$. This implies that the point $(a,q)$ lies on the curve $y=rP(x)$. 
Let $(a^{\prime},q^{\prime})\in\mathbb{N}^2$ be such that $q^{\prime}=rP(a^{\prime})$ and $a^{\prime}<a$. This in turn implies that $P(a)|qP(a^{\prime})$, but since $\gcd(P(a),q)=1$, it follows that $P(a)|P(a^{\prime})$ which is not true since $a^{\prime}<a$. Therefore, $(a,q)\in V(\mathbi{c})$, which implies that $\frac{a}{q}\in\mathcal{F}_{Q,V}$. The inclusion is strict because for example, if $P(x)=x(x+1)$, then clearly $\frac{1}{2}\in\mathcal{F}_{Q,V}$ as $(1,2)\in V((1,1))$ since $2=1\cdot P(1)$. However, $\frac{1}{2}\notin\mathcal{F}_{Q,P}$ since $\gcd(P(1),2)\ne 1$.

Our interest in this article lies in understanding the distribution of the sequence in $\mathcal{F}_{Q,P}$. In the first part, we consider the global distribution of the sequence. One aspect of studying the global distribution of a sequence is via equidistribution.  
While equidistribution modulo one is a qualitative asymptotic property, it is
natural to have a corresponding quantitative concept
which applies to finite sequences (or finite truncations of infinite sequences) such as the
concept of discrepancy of a sequence, which is defined as follows:

For any $\alpha\in[0,1]$, let $A(\alpha;{N})$ be the number of first $N$ terms of the sequence $(x_n)_{n=1}^{\infty}$ modulo one that do not exceed $\alpha$. Then the absolute discrepancy of the sequence $(x_n)_{n=1}^{\infty}$ is given by
\[D_{{N}}(x_1,\ldots, x_n)=\sup_{0\leq \alpha\leq 1}R_{{N}}(\alpha),\numberthis\label{D_1}\]
where
\[R_{{N}}(\alpha)=\left|\frac{A(\alpha;{N})}{{N}}-\alpha \right|.\numberthis\label{R_N}\]

 The classical work of Franel \cite{Franel} and Landau \cite{Landau} showed that the quantitative statement about the uniform distribution of Farey fractions and the Riemann hypothesis are equivalent. In particular, Franel proved that the supremum of the real part of the zeros of the Riemann zeta function is the infimum of $\theta$ for which the following estimate holds
\[\sum_{i=1}^{N(Q)}R_{N(Q)}^2(\gamma_i)=\BigO{Q^{-2+2\theta}}, \]
where $\gamma_i\in\mathcal{F}_Q,\ 1\leq i\leq N(Q)$ and $N(Q)=\#\mathcal{F}_Q$. Specifically, the Riemann hypothesis is equivalent to the asymptotic formula
\[\sum_{i=1}^{N(Q)}R_{N(Q)}^2(\gamma_i)=\BigO{Q^{-1+\epsilon}},\ \text{for all}\ \epsilon>0. \]
Landau \cite{Landau} gave a similar version by proving that the Riemann hypothesis is true if and only if, for all $\epsilon>0$,
\[\sum_{i=1}^{N(Q)}R_{N(Q)}(\gamma_i)=\BigO{Q^{1/2+\epsilon}}.\]
A lot of efforts have been made to prove the above estimates in terms of discrepancy. The first significant result concerning the distribution of Farey fractions was established by Erd\H{o}s et al. \cite{Erdos}. Neville \cite{Neville} proved that $D_{N(Q)}(\mathcal{F}_Q)=\BigO{\log Q/Q}$. Thereafter, it was improved by Niederreiter \cite{Niederreiter} to $D_{N(Q)}(\mathcal{F}_Q)\asymp{1/Q}$ for all $Q\geq 1$. A closed-form formula for the discrepancy was later established by \cite{Dress} who showed that $D_{N(Q)}(\mathcal{F}_Q)=1/Q$ holds for every $Q$. Discrepancy of Farey sequence with certain congruence constraints on denominators appeared in the works of Alkan et al. \cite{MR2273359, MR2275343} and Ledoan \cite{MR3871604} too. 
Our first main result establishes bounds on the discrepancy of the polynomial Farey sequence \eqref{F_Q}. Denote $D_{\mathcal{N}_{Q,P}}(\mathcal{F}_{Q,P})$ as the discrepancy of $(\mathcal{F}_{Q,P})_{Q\ge 1}$, where $\mathcal{N}_{Q,P}=\#\mathcal{F}_{Q,P}$. 


\begin{thm}\label{Disc}
Let $\nu\geq 2$ be an integer and let $\mathbi{c}=(c_\nu,c_{\nu-1},\ldots,c_{1})\in\mathbb{Z}^{\nu}$ be a fixed non-zero vector and $P(x)=c_{\nu}x^{\nu}+c_{\nu-1}x^{\nu-1}+\cdots+c_1x$ be a polynomial with non-zero discriminant. For all $Q\geq 1$, we have
    \[\frac{1}{Q}\ll D_{\mathcal{N}_{Q,P}}(\mathcal{F}_{Q,P})\ll{\frac{(\log Q)^J}{Q}},\]
  where $J$ is the number of distinct irreducible factors of the polynomial $P(x)\in\mathbb{Z}[x]$ and the implied constants depend on the polynomial $P(x)$.
\end{thm}

From Theorem \ref{Disc}, it follows that $R_{\mathcal{N}_{Q,P}}(\alpha)\to 0$ as $Q\to\infty$ for all $\alpha\in[0,1]$. Consequently, we obtain the following corollary.
\begin{cor}
  Let $\nu\geq 2$ be an integer and let $\mathbi{c}=(c_\nu,c_{\nu-1},\ldots,c_{1})\in\mathbb{Z}^{\nu}$ be a fixed non-zero vector and $P(x)=c_{\nu}x^{\nu}+c_{\nu-1}x^{\nu-1}+\cdots+c_1x$ be a polynomial with non-zero discriminant. Then the Farey sequence $(\mathcal{F}_{Q,P})_{Q\geq 1}$ is uniformly distributed modulo one.
\end{cor}

To understand the finer distribution of a sequence, such as randomness, local clustering, and periodic patterns of a sequence, quantities such as gap distribution, $k$-level correlation measure can be studied. 
These deal with the distribution on the scale of mean gap $1/N$. 
In the next section, we compute the 2-level or the pair correlation of polynomial Farey fractions. 

\subsection{Correlations of Farey fractions}
The motivation to study the correlations of sequences comes from applications in physics, where physicists study the spectra of high energies. Lately, there has been a lot of interest in these notions in applications in number theory, mathematical physics, and probability theory. It has particularly attracted significant interest in number theory following Montgomery \cite{Montgomery} and Hejhal's \cite{Hejhal} work on the correlations of the zeros of the Riemann zeta function, and Rudnick and Sarnak's \cite{Rudnick} work on the correlations of zeros of $L$-functions. Let ${F}$ be a finite set of $\mathcal{N}$ elements in the unit interval $[0,1]$ and for every interval $I\subset\mathbb{R}$, define
\[\mathcal{S}_{{F}}(I):=\frac{1}{\mathcal{N}}\#\left\{(a,b)\in{F}^2 : a\ne b, a-b\in\frac{1}{\mathcal{N}}I+\mathbb{Z}  \right\}. \]
The limiting pair correlation measure of an increasing sequence $({F}_n)_n$, for every interval $I$ is given (if it exists) by
\[\mathcal{S}(I)=\lim_{n\to\infty}\mathcal{S}_{{F}_n}(I).\]
If
\[\mathcal{S}(I)=\int_{I}g(x)dx, \numberthis\label{g(x)}\]
then $g$ is called the limiting pair correlation function of $({F}_n)_n$. The pair correlation is said to be Poissonian if $g(x)\equiv 1$.

For the usual Farey fractions, Boca and Zaharescu \cite{BocaF} proved that the pair correlation is non-Poissonian and it is given by 
\[g(\lambda)=\frac{6}{\pi^2\lambda^2}\sum_{1\leq k<\frac{\pi^2\lambda}{3}}\phi(k)\log\frac{\pi^2\lambda}{3k}\numberthis\label{g(lambda)}.\]
The formula in \eqref{g(lambda)} appeared in the main term of the second moment of a large sieve matrix \cite{Radziwil}. Results on the pair correlation for Farey sequences with divisibility restrictions on the Farey denominators were obtained in \cite{Siskaki} and \cite{Zaharescu}.
In \cite{Chahal}, the authors extended this study for Farey fractions with square-free denominators and proved that its pair correlation function exists and it is non-Poissonian. 
Our next main result concerns the pair correlation statistics for polynomial Farey fractions. 

\begin{thm}\label{thm2}
    Let $\nu\geq 2$ be an integer and let $P(x)=c_{\nu}x^{\nu}+c_{\nu-1}x^{\nu-1}+\cdots+c_1x\in\mathbb{Z}[x]$ be a polynomial with non-zero discriminant. Then $\limsup_{Q\to\infty}\mathcal{S}_{\mathcal{F}_{Q,P}}(\Lambda)$ is finite. Moreover, the $\limsup_{Q\to\infty}\mathcal{S}_{\mathcal{F}_{Q,P}}(\Lambda)$ is bounded by 
    \[\limsup_{Q\to\infty}\mathcal{S}_{\mathcal{F}_{Q,P}}(\Lambda)\ll\int_{0}^{\Lambda}\frac{1}{\lambda^{2}}\sum_{1\leq m<\frac{2\lambda}{\beta_{P}}}h(m)\log\left(\frac{2\lambda}{m\beta_{P}} \right) d\lambda \numberthis\label{Thm2}\]
for any $\Lambda > 0$, where $\beta_{P}=\prod_p \left(1-\frac{f_{P}(p)}{p^2}\right)$, 
 $h(m)=\sum_{\substack{n,\delta, d_1,d_2\geq 1\\n\delta d_1d_2=m}}nf_P(d_1)f_P(d_2)$ 
and \[f_{P}(p):=\#\{1\leq d\leq p : P(d)\equiv 0\pmod p \}.\numberthis\label{f_P}\]    
\end{thm}
For the specific polynomial $P(x)=x(x+1)$, we prove that the limiting pair correlation measure of the sequence $(\mathcal{F}_{Q,P})_Q$ exists and is non-Poissonian. Moreover, we derive an explicit formula for the pair correlation function in this case.
\begin{thm}\label{P(x)=x(x+1)}
    Let $P(x)=x(x+1)$ be a fixed polynomial.  Then the limiting pair correlation function of the sequence $(\mathcal{F}_{Q,P})_Q$ exists and is given by
    \[\mathfrak{g}_2(\lambda)=\frac{1}{\zeta(2)\lambda^2}\sum_{d_2,d_4=1}^{\infty}\frac{\mu(d_2)\mu(d_4)}{d_2d_4\phi(d_2)\phi(d_4)}\sum_{1\leq m<\frac{2\lambda}{\beta_P}}\mathfrak{h}(m)\log\left(\frac{2\lambda}{m\beta_P}\right), \]
    where
    \begin{align*}
        \mathfrak{h}(m)=&\sum_{\substack{n,\delta, d_1,d_3\geq 1\\n\delta d_1d_3=m\\\gcd\left(\frac{d_2}{dG_1},\frac{d_4}{DG_2}\right)=1}}n\mu(d_1)\mu(d_3)dDG_1G_2
         \prod_{p|\frac{d_2d_4}{dDG_1G_2}}\left(1+\frac{1}{p}\right)^{-1},
    \end{align*}
   and $d=\gcd(d_1,d_2),\ D=\gcd(d_3,d_4),\ G_1=\gcd\left(\delta,\frac{d_2}{d}\right)$,\ $G_2=\gcd\left(\delta,\frac{d_4}{D}\right)$. 
\end{thm}

In Figure 1, we compare the plots for the pair correlation function $\mathfrak{g}_2(\lambda)$, the below bound of the pair correlation function for the polynomial $P(x)=x(x^2+1)$, 
\[g_{P}(\lambda):=\frac{1}{22\lambda^{2}}\sum_{1\leq m<\frac{2\lambda}{\beta_{P}}}\left(\sum_{\substack{n,\delta, d_1,d_2\geq 1\\n\delta d_1d_2=m}}n(d_1d_2)^{0.01}\right)\log\left(\frac{2\lambda}{m\beta_{P}} \right), \numberthis\label{boundpair}\] 
and the pair correlation function $g(\lambda)$ of the classical Farey sequence : 
against various distributions: Poissonian distribution $(g_{P_o}= 1)$, and GUE distribution $(g_{\text{GUE}})$.
In \eqref{boundpair}, we have used Huxley's \cite{HuxleyM} bound for $f_P(n)$. Note that for $P(x)=x(x^2+1),  f_{P}(2)=2$; for other values, $f_{P}(p)=3$ if $p\equiv 1\pmod{4}$, and $f_{P}(p)=1$ if $p\equiv 3\pmod{4}$. 
\begin{figure}[ht]
\centering
\subfloat{
\includegraphics[width=12cm, height=6.5cm]{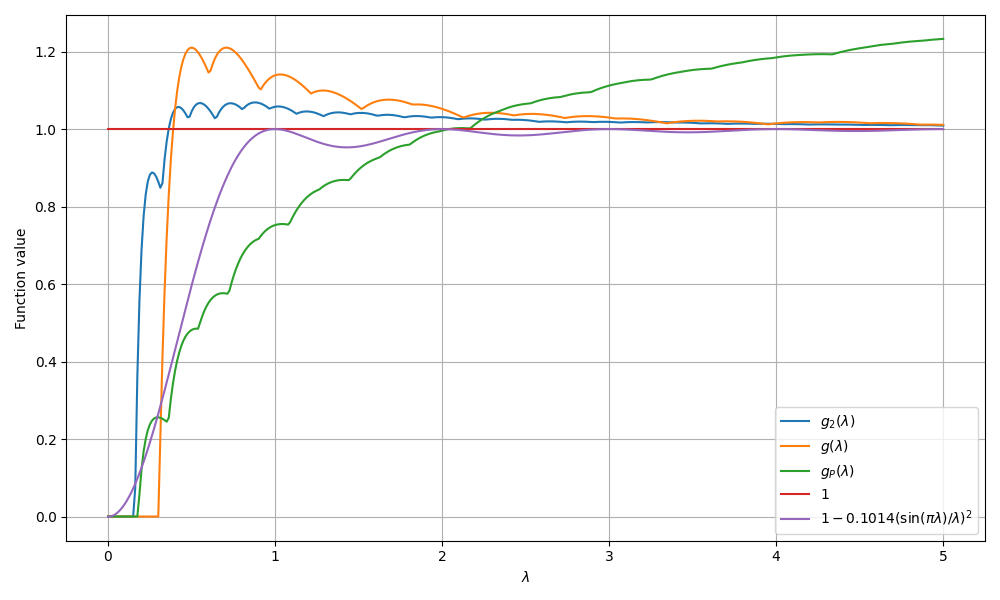}} 
\caption{The graphs of pair correlation functions $\mathfrak{g}_{2}(\lambda),\ {g}_{P}(\lambda),\ g(\lambda),\ g_{Po}(\lambda)\equiv 1$ and $g_{GUE}(\lambda)=1-\left(\frac{\sin \pi\lambda}{\pi\lambda} \right)^2$. }
\end{figure} 


Another question of significant interest is determining the values $\tau>0$ for which there are infinitely many primes that satisfy the Diophantine inequality
\[||\alpha p||<p^{-\tau+\epsilon} \numberthis\label{Diophantine}\]
for all $\epsilon>0$, where $||t||$ denotes the distance from the nearest integer to a real number $t$. Vinogradov \cite{Vinogradov} was the first to determine that $\tau=1/5$ is admissible and subsequently, his result was improved by various authors \cite{Harman, MR1367078, HB, Jia, MR1247382, MR1790174, MR1974139, VaughanR}. The inequality in \eqref{Diophantine} is equivalent to the existence of a fraction with prime denominator $a/p$ satisfying $|\alpha-\frac{a}{p}|<p^{-1-\tau+\epsilon}$. All the above results on $\tau$ are equivalent to quantitative statements on the gap between the Farey fractions with prime denominators. Motivated by this connection, we study the pair correlation statistics of polynomial Farey sequence with prime denominators. 
For each integer $Q$, let $\mathcal{B}_Q$ be a fixed subset of prime numbers that are less than or equal to $Q$. 
The polynomial Farey sequence with prime denominators is given by
\[\mathscr{M}_{\mathcal{B}_Q,P}:=\left\{\frac{a}{p}: 1\leq a\leq p\leq Q,\ \gcd(P(a),p)=1,\ p\in\mathcal{B}_Q \right\}. \numberthis\label{Farey prime}\]
A similar set as $\mathcal{B}_Q$ was considered by Xiao \cite{Xiao}, where he studied the pair correlation of the usual Farey fractions with denominators in that set.
\begin{thm}\label{thm3}
    The limiting pair correlation of the sequence $(\mathscr{M}_{\mathcal{B}_Q,P})_{Q\ge 1}$ exists as $Q\to\infty$ and is Poissonian if and only if $\sum_{p\in\mathcal{B}_Q}p^2=o\left((\#\mathscr{M}_{\mathcal{B}_Q,P})^2\right)$.
\end{thm}

If $P(x)=x$ and for each $Q$, $\mathcal{B}_Q$ is the set of prime numbers that are less than or equal to $Q$, and $\mathscr{M}_{Q}$ is the set of Farey fractions with prime denominators not exceeding $Q$, then by the prime number theorem and a partial summation formula using Theorem \ref{thm3}, we immediately deduce the following result of Xiong et al. \cite{Xiong}.
\begin{cor}\cite[Theorem 1]{Xiong}
  The limiting pair correlation of the sequence $(\mathscr{M}_{Q})_{Q\in\mathbb{N}}$ exists as $Q\to\infty$ and is constant equal to $1$.  
\end{cor}

Using Theorem \ref{thm3}, we next investigate the distribution of the fractions in $\mathscr{M}_{\mathcal{B}_Q,P}$ for different choices of the set $\mathcal{B}_Q$.
    
Let $a_1,\ldots, a_k$ be distinct positive even integers. A prime $p$ is said to be a prime $k$-tuple if $p+a_1,\ldots, p+a_k$ are all prime. Assuming the Hardy-Littlewood conjecture \cite{Hardy}, there are infinitely many prime $k$-tuples. Let $1<c<2$ be a fixed real number. A prime $p$ is said to be Piatetski-Shapiro prime if it is of the form $\lfloor n^c \rfloor$.
For more details on the Piatetski-Shapiro primes, one may refer to the well-known article of \cite{Piatetski}. A prime $p$ is said to be a Chen prime if $p+2$ is a product of atmost two primes (see \cite{Chen}).  
\begin{cor}\label{cor2}
 Let $c\in (1, 1.16)$ be a fixed real number. If $\mathcal{B}_Q^{(1)}$ and $\mathcal{B}_Q^{(2)}$ are the sets of Piatetski-Shapiro primes not exceeding $Q$ and Chen primes less than or equal to $Q$, respectively then the limiting pair correlation of the sequences $(\mathscr{M}_{\mathcal{B}_Q^{(1)},P})_{Q\in\mathbb{N}}$ and $(\mathscr{M}_{\mathcal{B}_Q^{(2)},P})_{Q\in\mathbb{N}}$ exists as $Q\to\infty$ and it is Poissonian. Moreover, if $\mathcal{B}_Q^{(3)}$ is the set of prime $k$-tuples less than or equal to $Q$, then under the Hardy-Littlewood prime $k$-tuple conjecture, the limiting pair correlation of the sequence $(\mathscr{M}_{\mathcal{B}_Q^{(3)},P})_{Q\in\mathbb{N}}$ exists as $Q\to\infty$ and is Poissonian.
\end{cor}
Let $q$ be a positive integer. For any choice of $\mathcal{D}\subseteq\{0, 1,\ldots,q-1\}$, let 
\[\mathbi{A}=\left\{\sum_{0\leq i\leq k}n_iq^i\ :\ n_i\in\{0,1,\ldots,q-1 \}\setminus\mathcal{D},\ k\geq 0 \right\}\]
be a set of integers with no digit in base $q$ in the set $\mathcal{D}$. A prime $p\in\mathbi{A}$ is said to be a prime with restricted digits. One may refer to the article of Maynard \cite{Maynard} for more on primes with restricted digits.
\begin{cor}\label{cor3}
   Let $q$ be a sufficiently large positive integer. For any choice of $\mathcal{D}\subseteq\{0, 1,\ldots,q-1\}$ with $|\mathcal{D}|\leq q^{23/80}$, let $\mathcal{B}_{Q}^{(4)}\subset\mathbi{A}$ be the set of primes that are less than or equal to $Q$ with restricted digits. Then the limiting pair correlation of the sequence $(\mathscr{M}_{\mathcal{B}_Q^{(4)},P})_{Q\in\mathbb{N}}$ exists as $Q\to\infty$ and is Poissonian. 
\end{cor}


In the next section, we study the problem of the distribution of polynomial Farey sequence along arithmetic progressions and bias over one from another. 
\subsection{Bias for Farey fractions}
As a consequence of Dirichlet's prime number theorem, the prime numbers in arithmetic progressions $a+nq$, with $a$ relatively prime to $q$, are evenly distributed. In particular, primes are asymptotically same among the residue classes modulo ${q}$ that are coprime to $q$.
 It was asserted by Chebyshev that there are more primes $\equiv 3\pmod 4$ than $\equiv 1\pmod 4$. Littlewood \cite{Littlewood} disproved Chebyshev's assertion by proving that the set of values of $x$ for which the difference $\pi(x;4,3)-\pi(x;4,1)$ is positive is unbounded, also there exists an unbounded set of values of $x$ for which the difference $\pi(x;4,3)-\pi(x;4,1)$ is negative. In this next part, we study Chebyshev's bias question for classical and polynomial Farey fractions. Let $q$ and $l$ be positive integers; we denote the number of polynomial Farey fractions with denominators in an arithmetic progression as
 \[S(Q;q,l):=\#\left\{\frac{a}{n}\in\mathcal{F}_{Q,P}\ |\ n\equiv l\pmod{q} \right\}.\numberthis\label{S_(Q)}\]
 We ask the following questions for Farey fractions analogous to prime number races.
 \begin{itemize}
     \item Does there exist positive integers $Q_0$, $l_1$, and $l_2$ with $l_1\not\equiv l_2\pmod{q}$ such that
 \[S(Q;q,l_1)>S(Q;q,l_2)\  \text{for all}\ Q>Q_0 ? \]
\item Are there arbitrarily large values of $Q$ for which $S(Q;q,l_1)<S(Q;q,l_2)$, and arbitrarily large values of $Q$ for which $S(Q;q,l_1)>S(Q;q,l_2)$? In other words, does the function $S(Q;q,l_1)-S(Q;q,l_2)$ change sign infinitely often?
 \end{itemize}
In here, we address the questions listed above. To state our results, we need the following conditions.

\noindent
\textbf{Haselgrove's condition for modulus $q$} \cite[p. 309]{Knapowski}: For all Dirichlet characters $\chi$ (mod $q$), $L(s,\chi)\ne 0$ for all $s\in(0,1)$.

Let $\nu,J\geq 1$ be integers and let $P(x)=c_{\nu}x^{\nu}+c_{\nu-1}x^{\nu-1}+\cdots+c_1x\in\mathbb{Z}[x]$ be a polynomial with non-zero discriminant and factorization
\[P(x)=\prod_{i=1}^Jm_i(x)^{e_i}, \numberthis\label{poly}\]
 where $m_i(x)\in\mathbb{Z}[x]$ are irreducible polynomials.
Let $K_i=\mathbb{Q}[x]/(m_i(x))$ be number fields with ring of integers $\mathcal{O}_{K_i}$. Let $\mathfrak{q}_i=\mathfrak{p}_1^{e_1}\cdots\mathfrak{p}_r^{e_r}\subset\mathcal{O}_{K_i}$ be an ideal with the unique prime factorization such that $p\in\mathfrak{p}_j$ for some $j$ and $p|q$. For $1\leq i\leq J$, let $\mathcal{L}_i(s,\chi^{\prime})$ be the Hecke $L$-functions associated with Hecke characters $\chi^{\prime}$ (mod ${\mathfrak{q}_i}$) of $\mathcal{O}_{K_i}$. We next state the analogous Haselgrove's condition for Hecke $L$-functions:

\noindent
\textbf{Haselgrove's condition for Hecke $L$-function mod $\mathfrak{q}$:} For all Hecke characters $\chi^{\prime}$ (mod $\mathfrak{q}$),\ $\mathcal{L}(s,\chi^{\prime})\ne 0$ for all $s\in(0,1)$.

Our main result for races of classical and polynomial Farey fractions is as follows.
\begin{thm}\label{thm4}
    Let $q\geq 2, l_1, l_2$ be positive integers such that $l_1\not\equiv l_2\pmod{q}$ and $\gcd(q,l_1l_2)=1$. Let $P(x)\in\mathbb{Z}[x]$ be as in \eqref{poly}. Assuming Haselgrove's condition for Hecke $L$-function $\mathcal{L}_i(s,\chi^{\prime})$ modulo $\mathfrak{q}_i$, the set of values of $Q$ for which the difference $S(Q;q,l_1)-S(Q;q,l_2)$ is strictly positive and the set of values of $Q$ for which the difference $S(Q;q,l_1)-S(Q;q,l_2)$ is strictly negative are unbounded. 
\end{thm}

\begin{rem}
    Note that Theorem \ref{thm4} does not give any information on the frequency of sign changes of the function $S(Q;q,l_1)-S(Q;q,l_2)$. We use a result of Kaczorowski and Wiertelak \cite[Lemma 3.1]{Kaczorowski} to gain information about the frequency of sign changes. We denote 
    \[\mathscr{A}(Q):=S(Q;q,l_1)-S(Q;q,l_2)\pm cQ^{\frac{1}{2}-\epsilon}, \]
  where $c$ is some positive constant and $\epsilon>0$ is arbitrarily small real number. In the definition of $\mathscr{A}(Q)$, we take the exponent of $Q$ to be $1/2-\epsilon$ in order to include the line $\Re(s)=1/2$ in our region enabling us to apply \cite[Lemma 3.1]{Kaczorowski}. Assuming Haselgrove's condition for Hecke $L$-function $\mathcal{L}_i(s,\chi^{\prime})$ modulo $\mathfrak{q}_i$, we apply \cite[Lemma 3.1]{Kaczorowski} to the function $\mathscr{A}(Q)$, and obtain a sequence $\{Q_i\}_{i=1}^{[\log T]}$ in the interval $(1,T]$ of length $\log T$ such that $sgn \mathscr{A}(Q_i)\ne sgn \mathscr{A}(Q_{i+1})$ and $|\mathscr{A}(Q_i)|>Q_i^{1/2-\epsilon}$. Hence, $\mathscr{A}(Q)$ has at least $\gg\log T$ oscillations of size $Q^{1/2-\epsilon}$, in the interval $(1,T]$.
\end{rem}

Moreover, we prove an $\Omega$-result for the error term in the asymptotic formula of $S(Q;q,l)$ in Proposition \ref{S(Q,q)}. 
\begin{thm}\label{thm5}
   Let $q\geq 2$ and $l$ be positive integers with $\gcd(q,l)=1$. Let $P(x)\in\mathbb{Z}[x]$ be as in \eqref{poly}. If $\Theta$ denotes the supremum of real parts of the zeros of the Hecke $L$-functions $\mathcal{L}_i(s,\chi)$ modulo $\mathfrak{q}_i$, then for any $\epsilon>0$, assuming the Haselgrove's condition, we have
   \[S(Q;q,l)-\frac{Q^2}{2\phi(q)}\prod_{p|q}\left(1-\frac{1}{p} \right)\prod_{p\nmid q}\left(1-\frac{f_P(p)}{p^2} \right)=\Omega_{\pm}(Q^{\Theta-\epsilon}), \]
   where $f_P(p)$ is as in \eqref{f_P}.
\end{thm}

\subsection{Organization}
The article is organized as follows. In Section 2, we provide some preliminary results required to prove our main results. We establish an asymptotic formula for $\#S(Q;q,l)$ and $\#\mathcal{F}_{Q,P}$ in Section 3. We discuss the discrepancy of polynomial Farey fractions in Section 4. Sections 5 and 6 are devoted to the proof of results concerning the correlation of polynomial Farey fractions. We establish the results on races of Farey fractions in Section 7. In an appendix in Section 8, we prove the results for the exponential sum twisted by the M\"{o}bius function and bound for twists of Merten's function by the arithmetic function $f_P(n)$ \eqref{f_P}.

\subsection{Notations}
 We write $f(x)=\BigO{g(x)}$ or equivalently $f(x)\ll g(x)$ if there exists a constant $C$ such that $|f(x)|\leq Cg(x)$ for all $x$. Further, a subscript of the form $\BigOqP{}$ or $\ll_q$ means that the implied constant $C$ may depend on $q$. And $f(x)=\Omega_{\pm}(g(x))$ means that $\limsup{f(x)}/{g(x)}>0$ and $\liminf{{f(x)}/{g(x)}}<0$. We say $f(x)=o(g(x))$ when $\lim_{x\to\infty}f(x)/g(x)=0$. We write $f(x)\sim g(x)$ when $\lim f(x)/g(x)\to 1$ as $x\to\infty$. In here, $f(x)\asymp g(x)$ means that there exist constants $C_1$ and $C_2$ such that $C_1g(x)\leq f(x)\leq C_2g(x)$. The symbol $L(s,\chi)$ denotes the Dirichlet L-function, and $\mathcal{L}(s,\chi^{\prime})$ stands for the Hecke L-function. Here, $\chi$ and $\chi^{\prime}$ are Dirichlet character modulo $q$ and Hecke character modulo $\mathfrak{q}$, respectively. Also, $e(x)=e^{2\pi ix}$, $s=\sigma+it\in\mathbb{C}$ is a complex number, and $\lfloor x\rfloor$ denotes the floor function. Throughout the paper, $\phi$ and $\mu$ denote the Euler's totient function and the M\"{o}bius function, respectively. In here, $\epsilon>0$ is an arbitrarily small positive real number and $\pi(x;q,a)$ denotes the number of primes $p\leq x$ such that $p\equiv a\pmod{q}$. For any set $A$, $\#A$ denotes the cardinality of the set $A$. And $\tau(n)$ denotes the number of positive divisors of $n$.


\subsection{Acknowledgements}
The first named author acknowledges the support from the University Grants Commission, Department of Higher Education, Government of India, under NTA Ref. no. 191620135578. The second named author acknowledges the support from Anusandhan National Research Foundation (ANRF) sanction no. CRG/2023/001743. The authors are also grateful to Greg Martin and Stephan Baier for generous discussions which inspired some of the contents of this work. The authors are grateful to the referees for providing several valuable suggestions and crucial corrections to the manuscript, which significantly improved the presentation.

\section{Preliminaries}
In this section, we review some results which are essential in proving our main results.
We recall a result proved in \cite{Chaubey} on an estimate of the average of the counting function for integer solutions $1\leq d\leq m$ of the polynomial congruence $P(d)\equiv 0\pmod m$.

\begin{prop}\label{sum_f}
    For a fixed non-zero vector $\mathbi{c}=(c_n,c_{n-1},\ldots,c_1)\in \mathbb{Z}^n$, let $P(x)=c_nx^n+c_{n-1}x^{n-1}+\cdots+c_1x$ be a polynomial with non-zero discriminant. If 
    $f_{P}(m)=\#\{1\leq d\leq m : P(d)\equiv 0\pmod m \}$, then as $x\to 
    \infty$, we have
    \[\sum_{m\leq x}f_{P}(m)\sim Cx(\log x)^{J-1},\]
    where $J\geq 2$ is the number of distinct irreducible factors of the polynomial $P(x)\in\mathbb{Z}[x].$
\end{prop}
\begin{proof}
    See \cite[Lemma 2.1]{Chaubey}.
\end{proof}

The Poisson summation formula is crucial in proving our result.
  \begin{prop}\cite[p. 538]{Vaughan}\label{Poisson} (Poisson's summation formula).
    Let $f\in L^1(\mathbb{R})$ and $\widehat{f}$ be the Fourier transform of $f$, then we have
    \[\sum_{n=-\infty}^{\infty}f(n)=\sum_{m=-\infty}^{\infty}\widehat{f}(m). \]
  \end{prop}  

Next, we state a result for counting the lattice points in a bounded domain with some coprimality restriction.
  \begin{prop}\cite[Lemma 2]{Boca}\label{lattice sum}
     Let $R>1$ be a real number. Let $\Omega \subset [1,R]^2$ be a bounded region, and $f$ is a continuously differentiable function on $\Omega$. Then
      \begin{align*}
      \sum_{\substack{(a,b)\in\Omega\cap \mathbb{Z}^2\\ \gcd(a,b)=1}}f(a,b)=&\frac{6}{\pi^2} \iint_{\Omega} f(x,y)dxdy \\ &+\BigO{\left(\left\|{\frac{\partial f}{\partial x}}\right\|_{\infty}+\left\|{\frac{\partial f}{\partial y}}\right\|_{\infty}\right)\text{Area}(\Omega)\log R+\|f\|_{\infty}(R+\text{length}(\partial \Omega)\log R)}. 
      \end{align*}
  \end{prop}
Moreover, we prove an asymptotic result to count the number of polynomial Farey fractions of order $Q$ with denominators in $\mathcal{B}_Q$.
\begin{prop}\label{M_BQ}
 Let $\nu\geq 1$ be an integer and $P(x)=c_{\nu}x^{\nu}+c_{\nu-1}x^{\nu-1}+\cdots+c_1x\in\mathbb{Z}[x]$ be a polynomial with non-zero discriminant. If $\mathscr{M}_{\mathcal{B}_Q,P}$ is as in \eqref{Farey prime}, then
    \[\mathscr{N}_{\mathcal{B}_Q,P}=\#\mathscr{M}_{\mathcal{B}_Q,P}=\sum_{\substack{p\in \mathcal{B}_Q}}p+\BigO{\#\mathcal{B}_Q}. \]
\end{prop}
\begin{proof}
We begin by expressing the number of fractions in $\mathscr{M}_{\mathcal{B}_Q,P}$ as a summation and then use the M\"{o}bius sum for the coprimality condition:
   \begin{align*}
   \#\mathscr{M}_{\mathcal{B}_Q,P}&=\sum_{p\in \mathcal{B}_Q}\sum_{\substack{a\leq p\\ \gcd(P(a),p)=1}}1
       =\sum_{\substack{p\in \mathcal{B}_Q\\}}\sum_{a\leq p}\sum_{\substack{d|P(a)\\d|p}}\mu(d)\\
        &=\sum_{\substack{p\in \mathcal{B}_Q}}\sum_{a\leq p}1-\sum_{p\in \mathcal{B}_Q}\sum_{\substack{a\leq p\\P(a)\equiv 0\pmod{p}}}1=\sum_{\substack{p\in \mathcal{B}_Q}}p-\sum_{\substack{p\in \mathcal{B}_Q}}f_P(p)
        =\sum_{\substack{p\in \mathcal{B}_Q}}p+\BigO{\#\mathcal{B}_Q}.
    \end{align*}
    In the last step, we applied Lagrange's theorem \cite[Chapter 7]{MR2445243} which implies that $P(x)$ has at most deg$(P)$ roots modulo $p$, thus $f_P(p)\leq \text{deg}(P)$.
\end{proof}

We also recall the well known Abel's summation formula.
\begin{prop}\cite[Theorem 4.2]{Apostol}\label{Abel}
    For any arithmetical function $a(n)$ let 
    $A(x)=\sum_{n\leq x}a(n), $
    where $A(x)=0$ if $x<1$. Assume f has a continuous derivative on the interval $[y,x]$, where $0<y<x$. Then we have
    \[\sum_{y<n\leq x}a(n)f(n)=A(x)f(x)-A(y)f(y)-\int_{y}^xA(t)f^{\prime}(t)dt. \]
\end{prop}

We next state a classical result of Landau on the singularities of the Mellin transform of a non-negative function.
\begin{prop}\cite{LandauE}\label{Landau}
    Let $A(x)$ be a real-valued function in one variable, and $A(x)$ does not change its sign for $x>x_0$, where $x_0$ is a sufficiently large real number. Suppose also for some real number $\beta<\gamma$, that Mellin transform $g(s):=\int_1^{\infty}A(x)x^{-s-1}dx$ is analytic for $\mathcal{R}(s)>\gamma$, can be analytically continued to the real segment $(\beta,\gamma]$. Then $g(s)$ represents an analytic function in the half plane $\mathcal{R}(s)>\beta$.
\end{prop}

\begin{prop}\cite[Lemma 15.1]{Vaughan}\label{omega}
    Suppose that $A(x)$ is a bounded Riemann integrable function in any finite interval $1\leq x\leq X$, and that $A(x)\geq 0$ for all $x>X_0$. Let $\sigma_c$ denote the infimum of those $\sigma$ for which $\int_{X_0}^{\infty}A(x)x^{-\sigma}dx<\infty$. Then the function
    \[F(s)=\int_{1}^{\infty}A(x)x^{-s}dx \]
    is analytic in the half plane $\sigma>\sigma_c$, but not at the point $s=\sigma_c$.
\end{prop}

We also use a version of the Poisson summation formula as below.
\begin{prop}\label{2.8}
 Let $\alpha,\beta\in\mathbb{R}$, and let $f$ be a smooth function on $\mathbb{R}$ such that Supp$(f)\subset(0,\Lambda)$. Then we have
 \[\sum_{n\in\mathbb{Z}}f(n+\alpha)e(\beta n)=\sum_{n\in\mathbb{Z}}\widehat{f}(n-\beta)e((n-\beta)\alpha), \]
 where $\widehat{f}$ is the Fourier transform of $f$.
\end{prop}
\begin{proof}
    Note that $f\in L^1(\mathbb{R})$. We define $g(x)=f(x+\alpha)e(\beta x)$. Clearly, $g$ is a smooth function and $g\in L^1(\mathbb{R})$, since $\int_{\mathbb{R}}|g(x)|dx=\int_{\mathbb{R}}|f(x)|dx<\infty$. The Fourier transform of $g$ is given by
    \begin{align*}
        \widehat{g}(x)&=\int_{\mathbb{R}}g(y)e(-xy)dy=\int_{\mathbb{R}}f(y+\alpha)e(-(x-\beta)y)dy\\&=\int_{\mathbb{R}}f(z)e(-(x-\beta)(z-\alpha))dz=\widehat{f}(x-\beta)e((x-\beta)\alpha).
    \end{align*}
We apply Proposition \ref{Poisson} to $g(x)=f(x+\alpha)e(\beta x)$. Therefore
\[\sum_{n\in\mathbb{Z}}f(n+\alpha)e(\beta n)=\sum_{n\in\mathbb{Z}}\widehat{f}(n-\beta)e((n-\beta)\alpha). \]
This completes the proof of Proposition \ref{2.8}.
\end{proof}

The following result on the lattice point counting problem, involving weight and coprimality conditions, is vital for proving the result on pair correlation measure.
\begin{lem}\label{lem2}
    Let $\Omega\subset[1,R]^2$ be a bounded region, and $f$ be a continuously differentiable function on $\Omega$. For any positive integers $r_1$ and $r_2$, we have
\[\sum_{\substack{(a,b)\in\Omega\cap\mathbb{Z}^2\\\gcd(a,r_1)=1=\gcd(b,r_2)\\\gcd(a,b)=1}}f(a,b)=\frac{1}{\zeta(2)}\prod_{p|r_1r_2}\left(1-\frac{1}{p^2} \right)^{-1}\prod_{p|r_1}\left(1-\frac{1}{p}\right)\prod_{p|r_2}\left(1-\frac{1}{p}\right)\iint_{\Omega}f(x,y)dxdy+E,\numberthis\label{2.3}\] 
     where
       \[ E\ll \left(\tau(r_1)\left\|{\frac{\partial f}{\partial x}}\right\|_{\infty}+\tau(r_2)\left\|{\frac{\partial f}{\partial y}}\right\|_{\infty}\right)\text{Area}(\Omega)\log^2 R+\|f\|_{\infty}(\tau(r_1)+\tau(r_2))R\log^2 R.\]
\end{lem}
\begin{proof}
We begin by removing the coprimality conditions on the left-hand side of \eqref{2.3} using M\"{o}bius summation. Therefore, we obtain  
   \begin{align*}
L:=\sum_{\substack{(a,b)\in\Omega\cap\mathbb{Z}^2\\\gcd(a,r_1)=1=\gcd(b,r_2)\\\gcd(a,b)=1}}f(a,b)&=\sum_{\substack{(a,b)\in\Omega\cap\mathbb{Z}^2\\\gcd(a,r_1)=1=\gcd(b,r_2)}}f(a,b)\sum_{\substack{d|a\\d|b}}\mu(d)=\sum_{d\leq R}\mu(d)\sum_{\substack{(a,b)\in\Omega\cap\mathbb{Z}^2\\\gcd(a,r_1)=1=\gcd(b,r_2)\\d|a, d|b}}f(a,b)\\
&=\sum_{\substack{d\leq R\\\gcd(d,r_1r_2)=1}}\mu(d)\sum_{\substack{(a_1,b_1)\in\frac{1}{d}\Omega\cap\mathbb{Z}^2\\\gcd(a_1,r_1)=1=\gcd(b_1,r_2)}}f(da_1,db_1)\\
&=\sum_{\substack{d\leq R\\\gcd(d,r_1r_2)=1}}\mu(d)\sum_{\substack{(a_1,b_1)\in\frac{1}{d}\Omega\cap\mathbb{Z}^2}}f(da_1,db_1)\sum_{\substack{d_1|a_1\\d_1|r_1}}\mu(d_1)\sum_{\substack{d_2|b_1\\d_2|r_2}}\mu(d_2)\\
&=\sum_{\substack{d\leq R\\\gcd(d,r_1r_2)=1}}\mu(d)\sum_{d_1|r_1}\mu(d_1)\sum_{d_2|r_2}\mu(d_2)\sum_{(a^{\prime},b^{\prime})\in\Delta\cap\mathbb{Z}^2}\mathfrak{f}(a^{\prime},b^{\prime}),\numberthis\label{l1}
\end{align*}
where $\mathfrak{f}(a^{\prime},b^{\prime})=f(dd_1a^{\prime},dd_2b^{\prime})$ and $\Delta=\left\{(x,y) : x\in\frac{1}{dd_1}[1,R], y\in\frac{1}{dd_2}[1,R] \right\}$. We use \cite[Lemma 2]{Boca} to estimate the inner-most sum in the above identity and obtain
\begin{align*}
   \sum_{(a^{\prime},b^{\prime})\in\Delta\cap\mathbb{Z}^2}\mathfrak{f}(a^{\prime},b^{\prime})=&\iint_{\Delta}\mathfrak{f}(x,y)dxdy\\ &+ \BigO{\left(\left\|{\frac{\partial \mathfrak{f}}{\partial x}}\right\|_{\infty}+\left\|{\frac{\partial \mathfrak{f}}{\partial y}}\right\|_{\infty}\right)\text{Area}(\Delta) +\|\mathfrak{f}\|_{\infty}(1+\text{length}(\partial \Delta))}\\
   =&\frac{1}{d^2d_1d_2}\iint_{\Omega}f(x,y)dxdy \\ &+ \BigO{\left(\frac{1}{dd_2}\left\|{\frac{\partial {f}}{\partial x}}\right\|_{\infty}+\frac{1}{dd_1}\left\|{\frac{\partial {f}}{\partial y}}\right\|_{\infty}\right)\text{Area}(\Omega) +\|{f}\|_{\infty}\frac{R}{d}\left(\frac{1}{d_1}+\frac{1}{d_2}\right)}.
\end{align*}
Inserting the above estimate into \eqref{l1}, we obtain
\begin{align*}
L=& \sum_{\substack{d\leq R\\\gcd(d,r_1r_2)=1}}\frac{\mu(d)}{d^2}\sum_{d_1|r_1}\frac{\mu(d_1)}{d_1}\sum_{d_2|r_2}\frac{\mu(d_2)}{d_2}\iint_{\Omega}f(x,y)dxdy+\BigO{\|{f}\|_{\infty}(\tau(r_1)+\tau(r_2)R\log^2R}\\ &+ \BigO{\left(\tau(r_1)\left\|{\frac{\partial {f}}{\partial x}}\right\|_{\infty}+\tau(r_2))\left\|{\frac{\partial {f}}{\partial y}}\right\|_{\infty}\right)\text{Area}(\Omega)\log^2R}\\
=&\frac{1}{\zeta(2)}\prod_{p|r_1r_2}\left(1-\frac{1}{p^2} \right)^{-1}\prod_{p|r_1}\left(1-\frac{1}{p} \right)\prod_{p|r_2}\left(1-\frac{1}{p} \right)\iint_{\Omega}f(x,y)dxdy+\BigO{\|{f}\|_{\infty}(\tau(r_1)+\tau(r_2))R\log^2R}\\ &+ \BigO{\left(\tau(r_1)\left\|{\frac{\partial {f}}{\partial x}}\right\|_{\infty}+\tau(r_2)\left\|{\frac{\partial {f}}{\partial y}}\right\|_{\infty}\right)\text{Area}(\Omega)\log^2R}.
\end{align*}
This completes the proof of Lemma \ref{lem2}.
\end{proof}

\section{Counting polynomial Farey fractions}
In this section, we prove an asymptotic formula for the number of polynomial Farey fractions with denominators in an arithmetic progression.
\begin{prop}\label{S(Q,q)}
  Let $\nu\geq 1$ be an integer and $P(x)=c_{\nu}x^{\nu}+c_{\nu-1}x^{\nu-1}+\cdots+c_1x\in\mathbb{Z}[x]$ be a polynomial with non-zero discriminant. Let $q$ and $l$ be positive integers with $\gcd(q,l)=1$. If $S(Q;q,l)$ is as in \eqref{S_(Q)}, then
    \[S(Q;q,l)=\frac{Q^2}{2\phi(q)}\prod_{p|q}\left(1-\frac{1}{p} \right)\prod_{p\nmid q}\left(1-\frac{f_P(p)}{p^2} \right)+\BigOqP{Q^{\frac{3}{2}+\epsilon}}, \]
    where $f_P(p)$ is as in \eqref{f_P}.
\end{prop}
\begin{proof}
    We begin with the sum
    \[S(Q;q,l)=\sum_{\substack{n\leq Q\\n\equiv l\pmod{q}}}\sum_{\substack{a\leq n\\\gcd(P(a), n)=1}}1.\numberthis\label{I3}\]
For fixed positive integers $q, l$ with $\gcd(q,l)=1$, in view of the identity
\[\frac{1}{\phi(q)}\sum_{\chi\pmod{q}}\chi(n\bar{l})=\left\{\begin{array}{cc}
   1  & \mbox{if} \ n\equiv l\pmod{q},\\
  0   & \mbox{otherwise,} 
\end{array}\right. \]
where $\chi$ is the Dirichlet character modulo $q$ and $\bar{l}$ is such that $l\bar{l}\equiv 1\pmod{q}$, the sum in \eqref{I3} can be written as
\begin{align*}
S(Q;q,l)&=\frac{1}{\phi(q)}\sum_{\chi\pmod q}\bar{\chi}(l)\sum_{n\leq Q}\chi(n)\sum_{\substack{a\leq n\\\gcd(P(a),n)=1}}1 =\frac{1}{\phi(q)}\sum_{\chi\pmod q}\bar{\chi}(l)\sum_{n\leq Q}\chi(n)\sum_{a\leq n}\sum_{\substack{d|P(a)\\d|n}}\mu(d)\\
&=\frac{1}{\phi(q)}\sum_{\chi\pmod q}\bar{\chi}(l)\sum_{n\leq Q}\chi(n)\sum_{d|n}\mu(d)\sum_{\substack{a\leq n\\ P(a)\equiv 0\pmod d}}1\\
&=\frac{1}{\phi(q)}\sum_{\chi\pmod q}\bar{\chi}(l)\sum_{n\leq Q}n\chi(n)\sum_{d|n}\frac{\mu(d)f_{P}(d)}{d}=\frac{1}{\phi(q)}\sum_{\chi\pmod q}\bar{\chi}(l)\sum_{n\leq Q}n\chi(n)K(n),\numberthis\label{I7}
\end{align*}
where $K(n)=\sum_{d|n}\frac{\mu(d)f_{P}(d)}{d}$.
  Note that the arithmetic function $n\chi(n)K(n)$ is multiplicative.  The Dirichlet series of $n\chi(n)K(n)$ is given by
  \begin{align*}
    F(s)&=\sum_{n=1}^{\infty}\frac{\chi(n)K(n)}{n^{s-1}}=\prod_{p}\left\{1+\left(1-\frac{f_{P}(p)}{p}\right)\left(\frac{\chi(p)}{p^{s-1}}+\frac{\chi(p^2)}{p^{2s-2}}+\cdots \right) \right\}\\
    &=\prod_{p}\left\{1+\left(1-\frac{f_{P}(p)}{p}\right)\frac{\chi(p)}{p^{s-1}}\cdot\frac{1}{1-\frac{\chi(p)}{p^{s-1}}} \right\}=L(s-1,\chi)\prod_{p}\left(1-\frac{\chi(p)f_{P}(p)}{p^s}\right), \numberthis\label{C3}
\end{align*}
  which is absolutely convergent for $\Re(s)>2$. Moreover, the product term on the far right side is absolutely convergent for $\Re(s)>1$. Thus, the Dirichlet series $F(s)$  has an analytic continuation to the half plane $\Re(s)>1$ except for a simple pole at $s=2$ in the case of principal Dirichlet character $\chi_0$. We use Perron's formula \cite[Theorem 2, p. 132]{Tenenbaum} for the Dirichlet series $F(s)$ with some fixed $\alpha=2+1/\log Q$ to obtain
  \[\sum_{n\leq Q}n\chi(n)K(n)=\frac{1}{2\pi i}\int_{\alpha-iT}^{\alpha+iT}\frac{F(s)Q^s}{s}ds+R(T),\numberthis\label{I1} \]
  where $R(T)\ll\frac{Q^{\alpha}}{T}\sum_{n=1}^{\infty}\frac{K(n)}{n^{\alpha}|\log(Q/n)|}$. Using the fact that $K(n)\ll n^{\epsilon}$, one can obtain bound for $R(T)$ in a similar way as in Davenport \cite{Davenport} (see p. 106-107). Therefore,
  \[R(T)\ll\frac{Q^{2+\epsilon}\log Q}{T}. \]
  To estimate the integral in \eqref{I1}, we shift the path of integration into a rectangular contour with line segments connecting the points $\alpha-iT, \alpha+iT, 3/2+\epsilon+iT$, and $3/2+\epsilon-iT$. We first consider the principal character $\chi_0$ (mod $q$). Applying Cauchy's residue theorem, we have
  \[\frac{1}{2\pi i}\int_{\alpha-iT}^{\alpha+iT}\frac{F(s)Q^s}{s}ds=\frac{Q^2}{2}\prod_{p|q}\left(1-\frac{1}{p} \right)\prod_{p\nmid q}\left(1-\frac{f_P(p)}{p^2}\right)+\sum_{j=1}^3I_j, \numberthis\label{I2}\]
  where $I_1$ and $I_3$ are the integrals along the horizontal line segments connecting the points $3/2+\epsilon+iT, \alpha+iT$ and $3/2+\epsilon-iT, \alpha-iT$, respectively, and $I_2$ is integral along vertical line $[3/2+\epsilon-iT, 3/2+\epsilon+iT]$. The first term in the above identity is due to the simple pole of the integrand at $s=2$. To estimate the integrals $I_1$ and $I_3$, we use the standard bounds for $\zeta(s)$ (see \cite[p. 47]{Titchmarsh}). Thus,
  \[I_1, I_3\ll_{q}
\frac{\log T}{T^{3/4}}\int_{3/2+\epsilon}^{2}Q^{\sigma}d\sigma+\frac{\log T}{T}\int_{2}^{\alpha}Q^{\sigma}d\sigma\ll_{q}\frac{Q^2\log T}{T^{3/4}\log Q}+\frac{Q^{\alpha}\log T}{T\log Q}. \]
  We use \cite[Lemma 2.2]{Bittu} to estimate the integral $I_2$ and obtain
  \[I_2\ll_{q}Q^{3/2+\epsilon}\int_{0}^{T}\frac{|\zeta(\frac{1}{2}+\epsilon+it)|}{|\frac{1}{2}+\epsilon+it|}dt\ll_{q}Q^{3/2+\epsilon}\log T. \]
  We next consider the case for the non-principal character $\chi\ne\chi_0$. We continue with the same contour defined above and use the bounds for $L(s-1,\chi)$ (see \cite{Kolesnik}). Therefore
  \[I_1, I_3\ll_{q}\frac{\log T}{T^{181/216}}\int_{3/2+\epsilon}^{2}Q^{\sigma}d\sigma+\frac{\log T}{T}\int_{2}^{\alpha}Q^{\sigma}d\sigma\ll_{q}\frac{Q^2\log T}{T^{181/216}\log Q}+\frac{Q^{\alpha}\log T}{T\log Q}, \]
  and
  \[I_2\ll_{q}Q^{3/2+\epsilon}\int_{0}^{T}\frac{|L(\frac{1}{2}+\epsilon+it)|}{|\frac{1}{2}+\epsilon+it|}dt\ll_{q}Q^{3/2+\epsilon}\log T. \]
  Collecting all the above estimates in \eqref{I1} and choosing $T=Q$, for $\chi=\chi_0$, we have
  \[\sum_{n\leq Q}n\chi_0(n)K(n)=\frac{Q^2}{2}\prod_{p|q}\left(1-\frac{1}{p} \right)\prod_{p\nmid q}\left(1-\frac{f_P(p)}{p^2}\right)+\BigOqP{Q^{\frac{3}{2}+\epsilon} },\numberthis\label{I4} \]
  and for $\chi\ne \chi_0$, we have
  \[\sum_{n\leq Q}n\chi(n)K(n)=\BigOqP{Q^{\frac{3}{2}+\epsilon}}.\numberthis\label{I5} \]
Inserting \eqref{I4} and \eqref{I5} into \eqref{I7} gives the required result.  
\end{proof}
The error term in Proposition \ref{S(Q,q)} is sharpened in terms of $\Omega-$result in Theorem \ref{thm5}. The exponent $3/2+\epsilon$ in the error term is sharpened to the supremum of the real part of zeros of Hecke $L$-function with a saving of $\epsilon$.
An immediate consequence of Proposition \ref{S(Q,q)} is the following corollary which is obtained upon taking $q=1$.
\begin{cor}\label{cor3.2}
   Let $\nu\geq 1$ be an integer and $P(x)=c_{\nu}x^{\nu}+c_{\nu-1}x^{\nu-1}+\cdots+c_1x\in\mathbb{Z}[x]$ be a polynomial with non-zero discriminant. If $\mathcal{F}_{Q,P}$ is as in \eqref{F_Q}, then
    \[\mathcal{N}_{Q,P}=\#\mathcal{F}_{Q,P}=\frac{Q^2}{2}\prod_p \left(1-\frac{f_{P}(p)}{p^2}\right)+\BigO{Q^{\frac{3}{2}+\epsilon}},\numberthis\label{N_Q}\] 
\end{cor}

\section{Discrepancy of polynomial Farey fractions}

In this section, we study the global statistics of polynomial Farey fractions. More generally, we analyze the discrepancy of fractions in  $\mathcal{F}_{Q, P}$.

\subsection{Proof of Theorem \ref{Disc}}
\begin{proof}
  Let $\alpha\in[0,1]$ be a real number. To establish the upper bound for the discrepancy, we write
    \[A(\alpha;\mathcal{N}_{Q,P})=\sum_{\gamma\in\mathcal{F}_{Q,P}\cap[0,\alpha]}1.\]
We next consider the far-right side of \eqref{R_N} to obtain
\begin{align*}
        A(\alpha;\mathcal{N}_{Q,P})-\alpha \mathcal{N}_{Q,P}&=\sum_{q\leq Q}\left(\sum_{\substack{a\leq \alpha q\\ \gcd(P(a),q)=1}}1-\alpha\sum_{\substack{a\leq  q\\ \gcd(P(a),q)=1}}1\right)
        =\sum_{q\leq Q}\left(\sum_{a\leq\alpha q}\sum_{\substack{d|P(a)\\d|q}}\mu(d)-\alpha\sum_{a\leq q}\sum_{\substack{d|P(a)\\d|q}}\mu(d) \right)\\
        &=\sum_{q\leq Q}\sum_{d|q}\mu(d)\left(\sum_{\substack{a\leq\alpha q\\d|P(a)}}1-\alpha\sum_{\substack{a\leq q\\d|P(a)}}1 \right).
        \end{align*}
    Since $f_{P}(d)$ counts the number of the solutions of the polynomial congruence $P(a)\equiv 0\pmod d$, so using this fact in the above estimate, we have
\[\sum_{\substack{a\leq\alpha q\\d|P(a)}}1=\left\lfloor \frac{\alpha q}{d} \right\rfloor f_{P}(d)+\BigO{f_P(d)}\ \text{and}\ \sum_{\substack{a\leq q\\d|P(a)}}1=\frac{q}{d}f_P(d). \]
Therefore, we have  
\begin{align*}
        |A(\alpha;\mathcal{N}_{Q,P})-\alpha \mathcal{N}_{Q,P}|&=
        \left|\sum_{q\leq Q}\sum_{d|q}\mu(d)\left(\left\lfloor \frac{\alpha q}{d} \right\rfloor f_{P}(d)+\BigO{f_P(d)}  - \frac{\alpha q}{d}  f_{P}(d) \right)\right|\\
        &=\left|\sum_{q\leq Q}\sum_{d|q}\mu(d)\left( \left(\frac{\alpha q}{d} +\BigO{1}\right) f_{P}(d) - \frac{\alpha q}{d}  f_{P}(d) \right)\right|
         \ll\sum_{q\leq Q}\sum_{d|q}f_{P}(d)\\&\ll\sum_{q\leq Q}\sum_{d\leq\frac{Q}{q}}f_{P}(d).\numberthis\label{D_2}
    \end{align*}
    By Proposition \ref{sum_f}, we have
    \[\sum_{d\leq \frac{Q}{q}}f_{P}(d)\sim \frac{Q}{q}\left(\log \frac{Q}{q}\right)^{J-1}, \]
 where $J$ is the number of distinct irreducible factors of $P(x)$. Thus, the above estimate with \eqref{D_2} yields
    \begin{align*}
        |A(\alpha;\mathcal{N}_{Q,P})-\alpha \mathcal{N}_{Q,P}|&\ll \sum_{q\leq Q}\frac{Q}{q}\left(\log \frac{Q}{q}\right)^{J-1}
        \ll Q\left(\log Q\right)^{J-1}\sum_{q\leq Q}\frac{1}{q}\ll Q\left(\log Q\right)^{J}.
    \end{align*}
Therefore
\begin{align*}
    R_{\mathcal{N}_{Q,P}}(\alpha)&=\frac{1}{\mathcal{N}_{Q,P}}\left|A(\alpha;\mathcal{N}_{Q,P})-\alpha \mathcal{N}_{Q,P}\right|
    \ll \frac{(\log Q)^J}{Q}\numberthis\label{R(N)}
\end{align*}
uniformly in $\alpha\in[0,1]$. Next, let $\epsilon>0$ be arbitrarily small and we take $\alpha=1/Q-\epsilon$ to obtain a lower bound for $D_{\mathcal{N}_{Q,P}}(\mathcal{F}_{Q,P})$. By the definition of $A(\alpha;\mathcal{N}_{Q,P})$, we have $A(1/Q-\epsilon;\mathcal{N}_{Q,P})=0$.
By \eqref{D_1} and \eqref{R_N}, we get
\[D_{\mathcal{N}_{Q,P}}(\mathcal{F}_{Q,P})\geq R_{\mathcal{N}_{Q,P}}(\alpha)=R_{\mathcal{N}_{Q,P}}\left(\frac{1}{Q}-\epsilon\right)=\frac{1}{Q}-\epsilon \]
for all $\epsilon>0$. Since $\epsilon>0$ is arbitrary, we can deduce this to
\[D_{\mathcal{N}_{Q,P}}(\mathcal{F}_{Q,P})\geq \frac{1}{Q}. \]
This completes the proof of Theorem \ref{Disc}.
\end{proof}

\section{Correlation of Polynomial Farey fractions}
This section is devoted to the proofs of Theorems \ref{thm2} and \ref{P(x)=x(x+1)}.
\subsection{Exponential sum over polynomial Farey fractions}
We begin by establishing results for the exponential sum over the Farey fractions in $\mathcal{F}_{Q,P}$.

\begin{lem}\label{exp sum}
Let $r$ and $\nu\geq 2$ be integers and let $P(x)=c_{\nu}x^{\nu}+c_{\nu-1}x^{\nu-1}+\cdots+c_1x\in\mathbb{Z}[x]$ be a polynomial with non-zero discriminant.
Then for any $\epsilon>0$, we have
    \[\sum_{\gamma\in\mathcal{F}_{Q, P}}e(r\gamma)= \sum_{\substack{q\leq Q\\q|r}}q\sum_{d\leq\frac{Q}{q} }\mu(d)\sum_{\substack{1\leq a\leq d\\P(a)\equiv 0\pmod {d}}}e\left(\frac{ra}{qd} \right), \]
    where $e(x)=\exp{(2\pi ix)}$. 
\end{lem}
\begin{proof}
   We have
    \begin{align*}
    \sum_{\gamma\in\mathcal{F}_{Q,P}}e(r\gamma)&=\sum_{q\leq Q}\sum_{\substack{1\leq a\leq q\\ \gcd(P(a),q)=1}}e\left(\frac{ar}{q} \right)
    =\sum_{q\leq Q}\sum_{1\leq a\leq q }e\left(\frac{ar}{q} \right)\sum_{\substack{d|P(a)\\d|q}}\mu(d)\\
    &=\sum_{d\leq Q}\mu(d)\sum_{\substack{q\leq Q\\d|q}}\sum_{\substack{1\leq a\leq q\\ d|P(a)}}e\left(\frac{ar}{q} \right)
    =\sum_{d\leq Q}\mu(d)\sum_{q\leq \frac{Q}{d}}\sum_{\substack{1\leq a\leq qd\\ P(a)\equiv 0\pmod d}}e\left(\frac{ar}{qd} \right).\numberthis\label{exp sum1}
    \end{align*} 
We first consider the innermost sum in the above equation
    \begin{align*}
        \sum_{\substack{1\leq a\leq qd\\ P(a)\equiv 0\pmod d}}e\left(\frac{ar}{qd} \right)&=\left(\sum_{\substack{1\leq a\leq d\\ P(a)\equiv 0\pmod d}}+\sum_{\substack{d< a\leq 2d\\ P(a)\equiv 0\pmod {d}}}+\cdots+\sum_{\substack{(q-1)d< a\leq qd\\ P(a)\equiv 0\pmod {d}}} \right)e\left(\frac{ar}{qd} \right)\\
        &=\sum_{j=0}^{q-1}\sum_{\substack{1\leq a\leq d\\P(a)\equiv 0\pmod {d}}}e\left(\frac{r(jd+a)}{qd} \right)=\sum_{j=0}^{q-1}e\left(\frac{rj}{q} \right)\sum_{\substack{1\leq a\leq d\\P(a)\equiv 0\pmod {d}}}e\left(\frac{ra}{qd} \right).\numberthis\label{poly exp}
    \end{align*}
 In view of the identity
 \[\sum_{n=1}^me(nl/m)=\left\{\begin{array}{cc}
   m,  & \mbox{if} \ m|l,\\
  0,   & \mbox{otherwise}, 
\end{array}\right.\numberthis\label{e1} \]
 the first sum on the right-hand side of \eqref{poly exp} is $q$ if $q|r$ and $0$ otherwise. So \eqref{poly exp} in conjunction with \eqref{exp sum1} gives
   \begin{align*}
       \sum_{\gamma\in\mathcal{F}_{Q,P}}e(r\gamma)&=\sum_{d\leq Q}\mu(d)\sum_{\substack{q\leq \frac{Q}{d}\\q|r}}q\sum_{\substack{1\leq a\leq d\\P(a)\equiv 0\pmod {d}}}e\left(\frac{ra}{qd} \right)=\sum_{\substack{q\leq Q\\q|r}}q\sum_{d\leq\frac{Q}{q} }\mu(d)\sum_{\substack{1\leq a\leq d\\P(a)\equiv 0\pmod {d}}}e\left(\frac{ra}{qd} \right).
   \end{align*}
This gives the required result. 
  \end{proof}
For the specific polynomial $P(x)=x(x+1)$, we can expand the Weyl sums using Dirichlet characters.
  \begin{lem}\label{lem1}
    For the polynomial $P(x)=x(x+1)$, and $r\in\mathbb{Z}$, we have
    \begin{align*}
\sum_{\gamma\in\mathcal{F}_{Q,P}}e(r\gamma)=&\sum_{d_1\leq Q}\mu(d_1)\sum_{d_2\leq Q}\frac{\mu(d_2)d_2}{\phi(d_2)\gcd(d_1,d_2)}\sum_{\substack{q\leq \frac{Q\gcd(d_1,d_2)}{d_1d_2}\\ \frac{qd_2}{\gcd(d_1,d_2)}|r }}q\\&+\sum_{q\leq Q}\sum_{d|q}\frac{\mu(d)}{\phi(d)}\sum_{\substack{\chi\pmod{d}\\\chi\ne\chi_0}}\bar{\chi}(-1)\sum_{\substack{a\leq q\\\gcd(a,q)=1}} \chi(a)e\left(\frac{ar}{q}\right),  \end{align*}
    where $e(x)=\exp{(2\pi ix)}$.
\end{lem}
\begin{proof}
We have
\begin{align*}
   \sum_{\gamma\in\mathcal{F}_{Q,P}}e(r\gamma)&=\sum_{\substack{q\leq Q}}\sum_{\substack{a\leq q\\\gcd(a(a+1),q)=1}}e\left(\frac{ar}{q}\right)
   =\sum_{\substack{q\leq Q}}\sum_{\substack{a\leq q\\\gcd(a,q)=1}}e\left(\frac{ar}{q}\right)\sum_{\substack{d_2|a+1\\d_2|q}}\mu(d_2)\\
   &=\sum_{q\leq Q}\sum_{d_2|q}\mu(d_2)\sum_{\substack{a\leq q\\\gcd(a,q)=1\\a\equiv -1\pmod{d_2}}}e\left(\frac{ar}{q}\right)  \\
   &=\sum_{q\leq Q}\sum_{d_2|q}\frac{\mu(d_2)}{\phi(d_2)}\sum_{\chi\pmod{d_2}}\bar{\chi}(-1)\sum_{\substack{a\leq q\\\gcd(a,q)=1}} \chi(a)e\left(\frac{ar}{q}\right).
\end{align*}
We separate the terms corresponding to the principal and non-principal Dirichlet characters and denote the sum on the right hand side of the above identity for principal and non-principal Dirichlet characters by $S(\chi_0)$ and $S(\chi)$, respectively.

For $\chi=\chi_0$, we have
\begin{align*}
    S(\chi_0)&=\sum_{q\leq Q}\sum_{d_2|q}\frac{\mu(d_2)}{\phi(d_2)}\sum_{\substack{a\leq q\\\gcd(a,q)=1}} \chi_0(a)e\left(\frac{ar}{q}\right)\\
    &=\sum_{q\leq Q}\sum_{d_2|q}\frac{\mu(d_2)}{\phi(d_2)}\sum_{\substack{a\leq q\\\gcd(a,q)=1}} e\left(\frac{ar}{q}\right)=\sum_{q\leq Q}\sum_{d_2|q}\frac{\mu(d_2)}{\phi(d_2)}\sum_{\substack{a\leq q}} e\left(\frac{ar}{q}\right)\sum_{\substack{d_1|a\\d_1|q}}\mu(d_1)\\
    &=\sum_{q\leq Q}\sum_{d_2|q}\frac{\mu(d_2)}{\phi(d_2)}\sum_{\substack{d_1|q}}\mu(d_1)\sum_{\substack{a\leq q\\d_1|a}} e\left(\frac{ar}{q}\right)=\sum_{d_1\leq Q}\mu(d_1)\sum_{q\leq \frac{Q}{d_1}}\sum_{d_2|qd_1}\frac{\mu(d_2)}{\phi(d_2)}\sum_{a\leq q}e\left(\frac{ar}{q}\right).
\end{align*}
We use \eqref{e1} to estimate the inner-most sum in the above identity. Therefore
 \begin{align*}
     S(\chi_0)&=\sum_{d_1\leq Q}\mu(d_1)\sum_{\substack{q\leq \frac{Q}{d_1}\\q|r}}q\sum_{d_2|qd_1}\frac{\mu(d_2)}{\phi(d_2)}=\sum_{d_1\leq Q}\mu(d_1)\sum_{d_2\leq Q}\frac{\mu(d_2)}{\phi(d_2)}\sum_{\substack{q\leq \frac{Q}{d_1}\\q|r,\ d_2|qd_1}}q\\&=\sum_{d_1\leq Q}\mu(d_1)\sum_{d_2\leq Q}\frac{\mu(d_2)}{\phi(d_2)}\sum_{\substack{q\leq \frac{Q}{d_1}\\q|r,\ \frac{d_2}{\gcd(d_1,d_2)}|q}}q  
     =\sum_{d_1\leq Q}\mu(d_1)\sum_{d_2\leq Q}\frac{\mu(d_2)d_2}{\phi(d_2)\gcd(d_1,d_2)}\sum_{\substack{q\leq \frac{Q\gcd(d_1,d_2)}{d_1d_2}\\\frac{qd_2}{\gcd(d_1,d_2)}|r}}q.
 \end{align*}
In the second last step, we used the fact that $n_1|n_2n_3$ if and only if $\frac{n_1}{\gcd(n_1,n_2)}|n_3$.

 For $\chi\ne\chi_0$, we have
\begin{align*}
    S(\chi)&=\sum_{q\leq Q}\sum_{d|q}\frac{\mu(d)}{\phi(d)}\sum_{\substack{\chi\pmod{d}\\\chi\ne\chi_0}}\bar{\chi}(-1)\sum_{\substack{a\leq q\\\gcd(a,q)=1}} \chi(a)e\left(\frac{ar}{q}\right).
\end{align*}
This completes the proof of Lemma \ref{lem1}.
\end{proof}
\subsection{Pair correlation measure}
  \subsubsection{Proof of Theorem \ref{thm2}}
To prove Theorem \ref{thm2}, we need to estimate, for any positive real number $\Lambda$, the quantity
\[\mathcal{S}_{\mathcal{F}_{Q,P}}(\Lambda)=\frac{1}{\mathcal{N}_{Q,P}}\#\{(\gamma_1,\gamma_2)\in \mathcal{F}_{Q,P}^2: \gamma_1\ne\gamma_2, \gamma_1-\gamma_2\in\frac{1}{\mathcal{N}_{Q,P}}(0,\Lambda)+\mathbb{Z}\},\numberthis\label{S lambda}\]
as $Q\to \infty.$ Let $H$ be any continuously differentiable function with Supp\ $H\subset(0,\Lambda).$ Define $f(y)=\sum_{n\in\mathbb{Z}}H(\mathcal{N}_{Q,P}(y+n)),\ y\in\mathbb{R},$ and $S_{Q,P}=\sum_{\substack{\gamma_1,\gamma_2\in \mathcal{F}_{Q,P}\\\gamma_1\ne\gamma_2 }}f(\gamma_1-\gamma_2)$. We use the approach of Boca and Zaharescu \cite{BocaF}, and change the problem of counting the tuples in \eqref{S lambda} into estimating the exponential sum over polynomial Farey fractions. We have
\[S_{Q,P}
     =\sum_{r\in \mathbb{Z}}c_r\left|\sum_{\gamma\in \mathcal{F}_{Q,P}}e(r\gamma)\right|^2,\numberthis\label{exp}\]
 where $c_r=\frac{1}{\mathcal{N}_{Q,P}}\widehat{H}\left(\frac{r}{\mathcal{N}_{Q,P}}\right)$ is Fourier coefficient of the Fourier series 
 $f(y)=\sum_{r\in \mathbb{Z}}c_re(ry)$
 and $\widehat{H}$ is the Fourier transform of $H.$
 We employ Lemma \ref{exp sum} in \eqref{exp} to obtain
\begin{align*}
    S_{Q,P}&= \sum_{r\in \mathbb{Z}}c_r\sum_{\substack{q_1\leq Q\\q_1|r}}q_1\sum_{d_1\leq\frac{Q}{q_1} }\mu(d_1)\sum_{\substack{1\leq a_1\leq d_1\\P(a_1)\equiv 0\pmod {d_1}}}e\left(\frac{ra_1}{q_1d_1} \right)\sum_{\substack{q_2\leq Q\\q_2|r}}q_2\sum_{d_2\leq\frac{Q}{q_2} }\mu(d_2)\sum_{\substack{1\leq a_2\leq d_2\\P(a_2)\equiv 0\pmod {d_2}}}e\left(-\frac{ra_2}{q_2d_2} \right)\\
    &=\sum_{d_1,d_2\leq{Q} }\mu(d_1)\mu(d_2)\sum_{\substack{q_1\leq \frac{Q}{d_1}\\q_2\leq \frac{Q}{d_2}}}q_1q_2\sum_{\substack{1\leq a_1\leq d_1\\1\leq a_2\leq d_2\\P(a_1)\equiv 0\pmod {d_1}\\P(a_2)\equiv 0\pmod {d_2}}}\sum_{n\in \mathbb{Z}}c_{n[q_1,q_2]}e\left(n[q_1,q_2]\left(\frac{a_1}{q_1d_1}-\frac{a_2}{q_2d_2}\right) \right),\numberthis\label{l13}
\end{align*}
where $[q_1,q_2]$ is the least common multiple of $q_1$ and $q_2$.
In order to estimate the inner-most sum, for each $y>0$ we consider the function
 \[H_y(x)=\frac{1}{y}H\left(\frac{x\mathcal{N}_{Q,P}}{y} \right),\ x\in\mathbb{R}.\numberthis\label{l11} \]
 Then
 \[\widehat{H}_y(z)=\frac{1}{\mathcal{N}_{Q,P}}\widehat{H}\left(\frac{yz}{\mathcal{N}_{Q,P}}\right).\numberthis\label{l12} \]
Since $c_r=\frac{1}{\mathcal{N}_{Q,P}}\widehat{H}\left(\frac{r}{\mathcal{N}_{Q,P}}\right)$, using \eqref{l12} the inner-most sum in \eqref{l13} can be expressed as
\begin{align*}
  \sum_{n\in \mathbb{Z}}c_{n[q_1,q_2]}e\left(n[q_1,q_2]\left(\frac{a_1}{q_1d_1}-\frac{a_2}{q_2d_2}\right) \right)=&\sum_{n\in\mathbb{Z}}\frac{1}{\mathcal{N}_{Q,P}}\widehat{H}\left(\frac{n[q_1,q_2]}{\mathcal{N}_{Q,P}} \right)e\left(n[q_1,q_2]\left(\frac{a_1}{q_1d_1}-\frac{a_2}{q_2d_2}\right) \right)\\
  =&\sum_{n\in\mathbb{Z}}\widehat{H}_{[q_1,q_2]}(n)e\left(n[q_1,q_2]\left(\frac{a_1}{q_1d_1}-\frac{a_2}{q_2d_2}\right) \right).\numberthis\label{l18}
\end{align*}
Next, we apply Proposition \ref{2.8} to the right-hand side of the above identity. We obtain
\begin{align*}
   \sum_{n\in\mathbb{Z}}\widehat{H}_{[q_1,q_2]}(n)e\left(n[q_1,q_2]\left(\frac{a_1}{q_1d_1}-\frac{a_2}{q_2d_2}\right) \right)&=\sum_{n\in\mathbb{Z}}H_{[q_1,q_2]}\left(n+\frac{a_1[q_1,q_2]}{q_1d_1}-\frac{a_2[q_1,q_2]}{q_2d_2}\right)\\
   &=\sum_{n\in\mathbb{Z}}\frac{1}{[q_1,q_2]}H\left(\frac{\mathcal{N}_{Q,P}}{[q_1,q_2]}\left(n+\frac{a_1[q_1,q_2]}{q_1d_1}-\frac{a_2[q_1,q_2]}{q_2d_2}\right)\right).\numberthis\label{l15}
\end{align*}
The above identity, in conjunction with \eqref{l18} and \eqref{l13}, yields
\begin{align*}
  S_{Q,P}&=\sum_{d_1,d_2\leq{Q} }\mu(d_1)\mu(d_2)\sum_{\substack{q_1\leq \frac{Q}{d_1}\\q_2\leq \frac{Q}{d_2}}}q_1q_2\sum_{\substack{1\leq a_1\leq d_1\\1\leq a_2\leq d_2\\P(a_1)\equiv 0\pmod {d_1}\\P(a_2)\equiv 0\pmod {d_2}}} \sum_{n\in\mathbb{Z}}\frac{1}{[q_1,q_2]}H\left(\frac{\mathcal{N}_{Q,P}}{[q_1,q_2]}\left(n+\frac{a_1[q_1,q_2]}{q_1d_1}-\frac{a_2[q_1,q_2]}{q_2d_2}\right)\right)\\
  &=\sum_{d_1,d_2\leq{Q} }\mu(d_1)\mu(d_2)\sum_{\substack{1\leq a_1\leq d_1\\1\leq a_2\leq d_2\\P(a_1)\equiv 0\pmod {d_1}\\P(a_2)\equiv 0\pmod {d_2}}}\sum_{\substack{q_1\leq \frac{Q}{d_1}\\q_2\leq \frac{Q}{d_2}}}\gcd(q_1,q_2) \sum_{n\in\mathbb{Z}}H\left(\frac{\mathcal{N}_{Q,P}}{[q_1,q_2]}\left(n+\frac{a_1[q_1,q_2]}{q_1d_1}-\frac{a_2[q_1,q_2]}{q_2d_2}\right)\right).
\end{align*}
Take $\gcd(q_1,q_2)=\delta$, so that $q_1=q_1^{\prime}\delta$ and $q_2=q_2^{\prime}\delta$ with $\gcd(q_1^{\prime},q_2^{\prime})=1$. Substituting this into above equation, we arrive at the following expression 
\begin{align*}
    S_{Q,P}&=\sum_{d_1,d_2\leq{Q} }\mu(d_1)\mu(d_2)\sum_{\substack{1\leq a_1\leq d_1\\1\leq a_2\leq d_2\\P(a_1)\equiv 0\pmod {d_1}\\P(a_2)\equiv 0\pmod {d_2}}}\sum_{\delta\leq \frac{Q}{\max(d_1,d_2)}}\delta\sum_{\substack{q_1^{\prime}\leq \frac{Q}{\delta d_1}\\q_2^{\prime}\leq \frac{Q}{\delta d_2}\\\gcd(q_1^{\prime},q_2^{\prime})=1}} \sum_{n\in\mathbb{Z}}H\left({\mathcal{N}_{Q,P}}\left(\frac{n}{q_1^{\prime}q_2^{\prime}\delta}+\frac{a_1}{q_1^{\prime}\delta d_1}-\frac{a_2}{q_2^{\prime}\delta d_2}\right)\right)\\
    &=\sum_{d_1,d_2\leq{Q} }\mu(d_1)\mu(d_2)\sum_{\substack{1\leq a_1\leq d_1\\1\leq a_2\leq d_2\\P(a_1)\equiv 0\pmod {d_1}\\P(a_2)\equiv 0\pmod {d_2}}}\sum_{\delta\leq \frac{Q}{\max(d_1,d_2)}}\delta\sum_{\substack{q_1^{\prime}\leq \frac{Q}{\delta d_1},\ q_2^{\prime}\leq \frac{Q}{\delta d_2}\\\frac{a_1}{q_1^{\prime} d_1}=\frac{a_2}{q_2^{\prime} d_2}\\ \gcd(q_1^{\prime},q_2^{\prime})=1}} \sum_{n\in\mathbb{Z}}H\left(\frac{n\mathcal{N}_{Q,P}}{q_1^{\prime}q_2^{\prime}\delta}\right)\\
    &\ll \sum_{d_1,d_2\leq{Q} }\sum_{\substack{1\leq a_1\leq d_1\\1\leq a_2\leq d_2\\P(a_1)\equiv 0\pmod {d_1}\\P(a_2)\equiv 0\pmod {d_2}}}\sum_{\delta\leq \frac{Q}{\max(d_1,d_2)}}\delta\sum_{\substack{q_1^{\prime}\leq \frac{Q}{\delta d_1},\ q_2^{\prime}\leq \frac{Q}{\delta d_2}\\ \gcd(q_1^{\prime},q_2^{\prime})=1}} \sum_{n\in\mathbb{Z}}H\left(\frac{n\mathcal{N}_{Q,P}}{q_1^{\prime}q_2^{\prime}\delta}\right)
    .\numberthis\label{S1}
\end{align*}
In the second last step, we used the fact that Supp $H\subset(0,\Lambda)$, since for large $Q$, either $\mathcal{N}_{Q,P}\left(\frac{a_1}{q_1^{\prime}\delta d_1}-\frac{a_2}{q_2^{\prime}\delta d_2}\right)>\Lambda$ or $\left(\frac{a_1}{q_1^{\prime}\delta d_1}-\frac{a_2}{q_2^{\prime}\delta d_2}\right)\leq 0$. Further, for a non-zero contribution from $H$, one must have
 $0<\frac{n\mathcal{N}_{Q,P}}{q_1^{\prime}q_2^{\prime}\delta}<\Lambda$
 which implies
 $\delta nd_1d_2<\frac{2\Lambda}{\beta_{P}}=:\mathscr{C}_{\Lambda}$, where $\beta_P=\prod_p \left(1-\frac{f_{P}(p)}{p^2}\right)$.
 By utilizing the aforementioned inequality and taking into account the observation that
 \[H\left(\frac{n\mathcal{N}_{Q,P}}{q_1^{\prime}q_2^{\prime}\delta} \right)=H\left(\frac{n\beta_{P}Q^2}{2q_1^{\prime}q_2^{\prime}\delta} \right)+\BigO{\frac{n}{q_1^{\prime}q_2^{\prime}\delta}Q^{\frac{3}{2}+\epsilon}}, \numberthis\label{H}\]
 the sum in \eqref{S1} can be expressed as
\begin{align*}
    S_{Q,P}
    &\ll\sum_{\substack{d_1,d_2,\delta,n\geq 1\\n\delta d_1d_2<\mathscr{C}_{\Lambda} }}\delta f_P(d_1)f_P(d_2)\sum_{\substack{q_1^{\prime}\leq \frac{Q}{\delta d_1},\ q_2^{\prime}\leq \frac{Q}{\delta d_2}\\\gcd(q_1^{\prime},q_2^{\prime})=1}} H\left(\frac{n\beta_PQ^2}{2q_1^{\prime}q_2^{\prime}\delta}\right).\numberthis\label{S2}
\end{align*}
In order to estimate the inner sum in \eqref{S2}, we employ Proposition \ref{lattice sum}, which gives an asymptotic result for counting the number of lattice points with some weight within a bounded region. We obtain
 \begin{align*}
 \sum_{\substack{q_{1}^{\prime}\leq \frac{Q}{\delta d_1}, q_2^{\prime}\leq \frac{Q}{\delta d_2}\\ \gcd(q_1^{\prime},q_2^{\prime})=1}}H\left(\frac{n\beta_{P}Q^2}{2q_1^{\prime}q_2^{\prime}\delta} \right)&= \frac{6}{\pi^2}\int_0^{\frac{Q}{\delta d_2}}\int_0^{\frac{Q}{\delta d_1}}H\left(\frac{n\beta_{P}Q^2}{2xy\delta} \right)dxdy+\BigO{Q\log Q}\\
 &= \frac{6Q^2}{\pi^2}\int_0^{\frac{1}{\delta d_2}}\int_0^{\frac{1}{\delta d_1}}H\left(\frac{n\beta_{P}}{2xy\delta} \right)dxdy+\BigO{Q\log Q}. \numberthis\label{S3}
 \end{align*}
 Further, we put $\lambda=\frac{n\beta_{P}}{2xy\delta}$, then the double integral transforms into the following expression
 \begin{align*}
     \int_0^{\frac{1}{\delta d_2}}\int_0^{\frac{1}{\delta d_1}}H\left(\frac{n\beta_{P}}{2xy\delta} \right)dxdy &= \frac{n\beta_P}{2\delta}\int_0^{\frac{1}{\delta d_1}}\int_{\frac{nd_2\beta_{P}}{2x}}^{\Lambda}\frac{H(\lambda)}{x{\lambda}^{2}}d\lambda dx
     = \frac{n\beta_P}{2\delta}\int_{\frac{\delta nd_1d_2\beta_{P}}{2}}^{\Lambda}\int_{\frac{ nd_2\beta_{P}}{2\lambda}}^{\frac{1}{\delta d_1}}\frac{H(\lambda)}{x{\lambda}^{2}}dxd\lambda \\
     &= \frac{n\beta_P}{2\delta}\int_{\frac{\delta nd_1d_2\beta_{P}}{2}}^{\Lambda}\frac{H(\lambda)}{\lambda^{2}}\log\left(\frac{2\lambda}{\delta nd_1d_2\beta_{P}} \right) d\lambda. \numberthis\label{S4}
 \end{align*}
 Thus, \eqref{S3} and \eqref{S4} in conjunction with \eqref{S2} gives 
 \begin{align*}
     S_{Q,P}&\ll {Q^2\beta_P}\sum_{\substack{d_1,d_2,\delta,n\geq 1\\n\delta d_1d_2<\mathscr{C}_{\Lambda} }}nf_P(d_1)f_P(d_2)\int_{\frac{\delta nd_1d_2\beta_{P}}{2}}^{\Lambda}\frac{H(\lambda)}{\lambda^{2}}\log\left(\frac{2\lambda}{\delta nd_1d_2\beta_{P}} \right) d\lambda\\
     &\ll Q^2\beta_P\sum_{1\leq m<\mathscr{C}_{\Lambda}}\int_{\frac{ m\beta_{P}}{2}}^{\Lambda}\frac{H(\lambda)}{\lambda^{2}}\log\left(\frac{2\lambda}{m\beta_{P}} \right) d\lambda\sum_{\substack{n,\delta, d_1,d_2\geq 1\\n\delta d_1d_2=m}}nf_P(d_1)f_P(d_2)\\
     &\ll Q^2\beta_P\int_0^{\Lambda}\frac{H(\lambda)}{\lambda^2}\sum_{1\leq m<\frac{2\lambda}{\beta_P}}h(m)\log\left(\frac{2\lambda}{m\beta_{P}} \right) d\lambda,\numberthis\label{l16}
 \end{align*}
 where
 \[h(m):=\sum_{\substack{n,\delta, d_1,d_2\geq 1\\n\delta d_1d_2=m}}nf_P(d_1)f_P(d_2). \]
Using \eqref{N_Q} and \eqref{l16}, we conclude that
\[\frac{S_{Q,P}}{\#\mathcal{F}_{Q,P}}\ll\int_0^{\Lambda}\frac{H(\lambda)}{\lambda^{2}}\sum_{1\leq m<\frac{2\lambda}{\beta_P}}h(m)\log\left(\frac{2\lambda}{m\beta_{P}} \right)d\lambda . \]
Next, we use the standard approximation argument to approximate the characteristic function of the interval $(0,\Lambda)$ from below and above by the smooth functions with compact support in $(0,\Lambda)$ and obtain 
\[ \mathcal{S}_{\mathcal{F}_{Q,P}}(\Lambda)\ll\int_0^{\Lambda}\frac{1}{\lambda^{2}}\sum_{1\leq m<\frac{2\lambda}{\beta_P}}h(m)\log\left(\frac{2\lambda}{m\beta_{P}} \right)d\lambda.\]
Therefore, 
\[\limsup_{Q\to\infty} \mathcal{S}_{\mathcal{F}_{Q,P}}(\Lambda)\ll\int_0^{\Lambda}\frac{1}{\lambda^{2}}\sum_{1\leq m<\frac{2\lambda}{\beta_P}}h(m)\log\left(\frac{2\lambda}{m\beta_{P}} \right)d\lambda.\]
Note that $h(m)<\infty$ for every $m$. Since $1\leq m<\frac{2\lambda}{\beta_{P}}$, it follows that for every $\Lambda>0$, the sum on the right hand side in the above inequality has only finitely many terms and so is finite. Therefore, $\limsup_{Q\to\infty} \mathcal{S}_{\mathcal{F}_{Q,P}}(\Lambda)$ is finite. This completes the proof of Theorem \ref{thm2}.

\subsubsection{Proof of Theorem \ref{P(x)=x(x+1)}}
To prove Theorem \ref{P(x)=x(x+1)}, we need to estimate, for any positive real number $\Lambda$ and polynomial $P(x)=x(x+1)$, the quantity in \eqref{S lambda}.
As in Theorem \ref{thm2}, we change the problem of counting the tuples in \eqref{S lambda} into estimating the exponential sum over polynomial Farey fractions. We have
\[S_{Q,P}
     =\sum_{r\in \mathbb{Z}}c_r\left|\sum_{\gamma\in \mathcal{F}_{Q,P}}e(r\gamma)\right|^2,\numberthis\label{e2}\]
 where $c_r=\frac{1}{\mathcal{N}_{Q,P}}\widehat{H}\left(\frac{r}{\mathcal{N}_{Q,P}}\right)$ is Fourier coefficient of the Fourier series 
 $f(y)=\sum_{r\in \mathbb{Z}}c_re(ry)$
 and $\widehat{H}$ is the Fourier transform of $H$.
 We employ Lemma \ref{lem1} in the above identity to obtain   
\begin{align*}
    S_{Q,P}=&\sum_{r\in\mathbb{Z}}c_r\left(\sum_{d_1\leq Q}\mu(d_1)\sum_{d_2\leq Q}\frac{\mu(d_2)d_2}{d\phi(d_2)}\sum_{\substack{q_1\leq \frac{dQ}{d_1d_2}\\ \frac{q_1d_2}{d}|r }}q_1+\sum_{\substack{d_1^{\prime}|q_1^{\prime}\\ q_1^{\prime}\leq Q}}\frac{\mu(d_1^{\prime})}{\phi(d_1^{\prime})}\sum_{\substack{\chi\pmod{d_1^{\prime}}\\\chi\ne\chi_0}}\bar{\chi}(-1)\sum_{\substack{a_1\leq q_1^{\prime}\\\gcd(a_1,q_1^{\prime})=1}} \chi(a_1)e\left(\frac{a_1r}{q_1^{\prime}}\right) \right) \\ &\times
    \left(\sum_{d_3\leq Q}\mu(d_3)\sum_{d_4\leq Q}\frac{\mu(d_4)d_4}{D\phi(d_4)}\sum_{\substack{q_2\leq \frac{QD}{d_3d_4}\\ \frac{q_2d_4}{D}|r }}q_2+\sum_{\substack{d_2^{\prime}|q_2^{\prime}\\q_2^{\prime}\leq Q}}\frac{\mu(d_2^{\prime})}{\phi(d_2^{\prime})}\sum_{\substack{\chi\pmod{d_2^{\prime}}\\\chi\ne\chi_0}}\bar{\chi}(-1)\sum_{\substack{a_2\leq q_2^{\prime}\\\gcd(a_2,q_2^{\prime})=1}} \chi(a_2)e\left(\frac{-a_2r}{q_2^{\prime}}\right) \right),
    \end{align*}
    where $d=\gcd(d_1,d_2)$ and $ D=\gcd(d_3,d_4)$.
    \begin{align*}
    S_{Q,P}=&\sum_{d_1,d_3\leq Q}\mu(d_1)\mu(d_3)\sum_{d_2, d_4\leq Q}\frac{\mu(d_2)\mu(d_4)d_2d_4}{\phi(d_2)\phi(d_4)dD}\sum_{\substack{q_1\leq \frac{Qd}{d_1d_2}\\ q_2\leq \frac{QD}{d_3d_4} }}q_1q_2\sum_{n\in\mathbb{Z}}c_{\left[\frac{q_1d_2}{d},\frac{q_2d_4}{D}\right]n}\\
    &+\sum_{d_1\leq Q}\mu(d_1)\sum_{d_2\leq Q}\frac{\mu(d_2)d_2}{d\phi(d_2)}\sum_{\substack{q_1\leq \frac{Qd}{d_1d_2}}}q_1\sum_{q_2^{\prime}\leq Q}\sum_{d_2^{\prime}|q_2^{\prime}}\frac{\mu(d_2^{\prime})}{\phi(d_2^{\prime})}\sum_{\substack{\chi\pmod{d_2^{\prime}}\\\chi\ne\chi_0}}\bar{\chi}(-1)\sum_{\substack{a_2\leq q_2^{\prime}\\\gcd(a_2,q_2^{\prime})=1}} \chi(a_2)\\
    &\times\sum_{n\in\mathbb{Z}}c_{\frac{q_1d_2n}{d}}e\left(-\frac{a_2q_1d_2n}{q_2^{\prime}d}\right)+\sum_{d_3\leq Q}\mu(d_3)\sum_{d_4\leq Q}\frac{\mu(d_4)d_4}{\phi(d_4)D}\sum_{\substack{q_2\leq \frac{QD}{d_3d_4} }}q_2\sum_{q_1^{\prime}\leq Q}\sum_{d_1^{\prime}|q_1^{\prime}}\frac{\mu(d_1^{\prime})}{\phi(d_1^{\prime})}\\&\times\sum_{\substack{\chi\pmod{d_1^{\prime}}\\\chi\ne\chi_0}}\bar{\chi}(-1)\sum_{\substack{a_1\leq q_1^{\prime}\\\gcd(a_1,q_1^{\prime})=1}} \chi(a_1)\sum_{n\in\mathbb{Z}}c_{\frac{q_2d_4n}{D}}e\left(\frac{a_1q_2d_4n}{q_1^{\prime}D}\right)+   \sum_{q_1^{\prime}\leq Q}\sum_{d_1^{\prime}|q_1^{\prime}}\frac{\mu(d_1^{\prime})}{\phi(d_1^{\prime})}\sum_{\substack{\chi\pmod{d_1^{\prime}}\\\chi\ne\chi_0}}\\
    &\times\bar{\chi}(-1)\sum_{\substack{a_1\leq q_1^{\prime}\\\gcd(a_1,q_1^{\prime})=1}} \chi(a_1)\sum_{q_2^{\prime}\leq Q}\sum_{d_2^{\prime}|q_2^{\prime}}\frac{\mu(d_2^{\prime})}{\phi(d_2^{\prime})}\sum_{\substack{\chi\pmod{d_2^{\prime}}\\\chi\ne\chi_0}}\bar{\chi}(-1)\sum_{\substack{a_2\leq q_2^{\prime}\\\gcd(a_2,q_2^{\prime})=1}} \chi(a_2)\sum_{r\in\mathbb{Z}}c_re\left(\frac{a_1r}{q_1^{\prime}}-\frac{a_2r}{q_2^{\prime}}\right),\numberthis\label{SQP}
\end{align*}
where $[a,b]$ is the least common multiple of $a$ and $b$. We estimate the sum of Fourier coefficients in the first term on the right-hand side of the above identity using \eqref{l11}, \eqref{l12}, and Proposition \ref{Poisson}. Therefore, we have
\[\sum_{n\in\mathbb{Z}}c_{\left[\frac{q_1d_2}{d},\frac{q_2d_4}{D}\right]n}
    =\sum_{n\in\mathbb{Z}}\frac{1}{\left[\frac{q_1d_2}{d},\frac{q_2d_4}{D}\right]}H\left(\frac{n\mathcal{N}_{Q,P}}{\left[\frac{q_1d_2}{d},\frac{q_2d_4}{D}\right]} \right). \numberthis\label{l14}\]
In order to estimate the other inner-most sums containing the Fourier coefficients, we use \eqref{l11}, \eqref{l12} and Proposition \ref{2.8} to obtain
  \begin{align*}
    \sum_{n\in\mathbb{Z}}c_{\frac{q_1d_2n}{d}}e\left(\frac{-a_2q_1d_2n}{q_2^{\prime}d}\right)&=\sum_{n\in\mathbb{Z}}\frac{1}{\mathcal{N}_{Q,P}}\widehat{H}\left(\frac{q_1d_2n}{d\mathcal{N}_{Q,P}} \right)e\left(-\frac{a_2q_1d_2n}{q_2^{\prime}d}\right)=\sum_{n\in\mathbb{Z}}\widehat{H}_{\frac{q_1d_2}{d}}\left(n \right)e\left(\frac{-a_2q_1d_2n}{q_2^{\prime}d}\right)\\&=\sum_{n\in\mathbb{Z}}H_{\frac{q_1d_2}{d}}\left(n-\frac{a_2q_1d_2}{q_2^{\prime}d} \right)=\sum_{n\in\mathbb{Z}}\frac{d}{q_1d_2}H\left(\frac{d\mathcal{N}_{Q,P}}{q_1d_2}\left(n-\frac{a_2q_1d_2}{q_2^{\prime}d} \right)\right).\numberthis\label{cd}
  \end{align*} 
  Similarly, we have
  \[\sum_{n\in\mathbb{Z}}c_{\frac{q_2d_4n}{D}}e\left(\frac{a_1q_2d_4n}{q_1^{\prime}D}\right)=\sum_{n\in\mathbb{Z}}\frac{D}{q_2d_4}H\left(\frac{D\mathcal{N}_{Q,P}}{q_2d_4}\left(n+\frac{a_1q_2d_4}{q_1^{\prime}D} \right)\right) \numberthis\label{cD}\]
  and
  \[\sum_{r\in\mathbb{Z}}c_re\left(\frac{a_1r}{q_1^{\prime}}-\frac{a_2r}{q_2^{\prime}}\right)=\sum_{r\in\mathbb{Z}}H\left(\mathcal{N}_{Q,P}\left(r+\frac{a_1}{q_1^{\prime}}-\frac{a_2}{q_2^{\prime}}\right) \right). \numberthis\label{cq}\]
    Note that $q_1\leq \frac{dQ}{d_1d_2},\ q_2\leq \frac{DQ}{d_3d_4}$ and $\mathcal{N}_{Q,P}\sim cQ^2$. Given that Supp $H\subset(0,\Lambda)$ such that $Q>\Lambda$, for sufficiently large $Q$, we have  
    $H\left(\frac{d\mathcal{N}_{Q,P}}{q_1d_2}\left(n-\frac{a_2q_1d_2}{dq_2^{\prime}} \right)\right)=0$, $H\left(\frac{D\mathcal{N}_{Q,P}}{q_2d_4}\left(n+\frac{a_1q_2d_4}{Dq_1^{\prime}} \right)\right)=0$ and $H\left(\mathcal{N}_{Q,P}\left(r+\frac{a_1}{q_1^{\prime}}-\frac{a_2}{q_2^{\prime}}\right) \right)=0$. This leads to the vanishing of the sums in \eqref{cd}, \eqref{cD}, and \eqref{cq} reducing \eqref{SQP} to the following identity:
\begin{align*}
    S_{Q,P}=&\sum_{d_1,d_3\leq Q}\mu(d_1)\mu(d_3)\sum_{d_2, d_4\leq Q}\frac{\mu(d_2)\mu(d_4)d_2d_4}{\phi(d_2)\phi(d_4)dD}\sum_{\substack{q_1\leq \frac{Qd}{d_1d_2}\\q_2\leq \frac{DQ}{d_3d_4}  }}q_1q_2\sum_{n\in\mathbb{Z}}\frac{1}{\left[\frac{q_1d_2}{d},\frac{q_2d_4}{D}\right]}H\left(\frac{n\mathcal{N}_{Q,P}}{\left[\frac{q_1d_2}{d},\frac{q_2d_4}{D}\right]} \right)\\
    =&\sum_{d_1,d_2,d_3,d_4\leq Q}\frac{\mu(d_1)\mu(d_2)\mu(d_3)\mu(d_4)}{\phi(d_2)\phi(d_4)}\sum_{\substack{q_1\leq \frac{dQ}{d_1d_2}\\q_2\leq \frac{DQ}{d_3d_4}  }}\gcd\left(\frac{q_1d_2}{d},\frac{q_2d_4}{D}\right)\sum_{n\in\mathbb{Z}}H\left(\frac{n\mathcal{N}_{Q,P}}{\left[\frac{q_1d_2}{d},\frac{q_2d_4}{D}\right]} \right).
\end{align*}
Denote $\gcd\left(\frac{q_1d_2}{d},\frac{q_2d_4}{D}\right)=\delta$. Note that $\gcd\left(\frac{q_1d_2}{d},\frac{q_2d_4}{D}\right)=\delta$ if and only if $\gcd\left(\frac{q_1d_2}{d\delta},\frac{q_2d_4}{D\delta}\right)=1$. Furthermore, $\delta|\frac{q_1d_2}{d}$ if and only if $\frac{\delta}{G_1}|q_1$ and $\delta|\frac{q_2d_4}{D}$ if and only if $\frac{\delta}{G_2}|q_2$, where $G_1:=\gcd\left(\delta,\frac{d_2}{d} \right)$ and $G_2:=\gcd\left(\delta,\frac{d_4}{D} \right)$. That is, we have $q_1=\frac{q_1^{\prime}\delta}{G_1}$ and $q_2=\frac{q_2^{\prime}\delta}{G_2}$ for some $q_1^{\prime}, q_2^{\prime}\in\mathbb{Z}$. With this reduction, we have
\begin{align*}
  S_{Q,P}=&\sum_{\substack{d_1,d_2,d_3,d_4\leq Q \\ \delta\leq \frac{Q}{\max(d_1,d_3)}}}\frac{\delta\mu(d_1)\mu(d_2)\mu(d_3)\mu(d_4)}{\phi(d_2)\phi(d_4)}\sum_{\substack{n\in\mathbb{Z} \\ (q_1^{\prime}, q_2^{\prime})\in\mathfrak{S}}} H\left(\frac{ndDG_1G_2\mathcal{N}_{Q,P}}{q_1^{\prime}q_2^{\prime}\delta d_2d_4} \right),\numberthis\label{l2}
\end{align*}
where $\mathfrak{S}:=\left\{q_1^{\prime}\leq \frac{QdG_1}{\delta d_1d_2}, q_2^{\prime}\leq \frac{DQG_2}{\delta d_3d_4}: \gcd\left(\frac{q_1^{\prime}d_2}{dG_1},\frac{q_2^{\prime}d_4}{DG_2} \right)=1\right\}.$
We use the fact that Supp $H\subset(0,\Lambda)$ and \eqref{N_Q}. For a non-zero contribution from $H$, one must have
 $0<\frac{ndDG_1G_2\mathcal{N}_{Q,P}}{q_1^{\prime}q_2^{\prime}\delta d_2d_4}<\Lambda$
 which implies
 $\delta nd_1d_3<\frac{2\Lambda}{\beta_{P}}=:\mathscr{C}_{\Lambda}$, where $\beta_P=\prod_p \left(1-\frac{f_{P}(p)}{p^2}\right)$.
 By utilizing the aforementioned inequality and taking into account Corollary \ref{cor3.2}, the sum in \eqref{l2} becomes
 \begin{align*}
     S_{Q,P}=&\sum_{\delta nd_1d_3<\frac{2\Lambda}{\beta_P}}\delta\mu(d_1)\mu(d_3)\sum_{\substack{d_2,d_4\leq Q\\\gcd\left(\frac{d_2}{dG_1},\frac{d_4}{DG_2}\right)=1}}\frac{\mu(d_2)\mu(d_4)}{\phi(d_2)\phi(d_4)}\sum_{(q_1^{\prime}, q_2^{\prime})\in\mathfrak{S}_1}H\left(\frac{ndDG_1G_2\beta_PQ^2}{2q_1^{\prime}q_2^{\prime}\delta d_2d_4} \right)
     +\BigO{Q^{\frac{3}{2}+\epsilon}}. \numberthis\label{l3}
 \end{align*}
 where $\mathfrak{S}_1:=\left\{q_1^{\prime}\leq \frac{dQG_1}{\delta d_1d_2}, q_2^{\prime}\leq \frac{DQG_2}{\delta d_3d_4}:\gcd\left(q_1^{\prime},\frac{d_4}{DG_2} \right)=\gcd(q_1^{\prime},q_2^{\prime})=\gcd\left(q_2^{\prime},\frac{d_2}{dG_1} \right)=1\right\}.$
 We next employ Lemma \ref{lem2} to estimate the innermost sum on the right-hand side of \eqref{l3}. Denoting this sum by $S^{(1)}$, we have
 \begin{align*}
     S^{(1)}=&\frac{1}{\zeta(2)}\prod_{p|\frac{d_2d_4}{dDG_1G_2}}\left(1-\frac{1}{p^2}\right)^{-1}\prod_{p|\frac{d_2}{dG_1}}\left(1-\frac{1}{p}\right)\prod_{p|\frac{d_4}{DG_2}}\left(1-\frac{1}{p}\right)\\ &\times\int_0^{\frac{DQG_2}{\delta d_3d_4}}\int_0^{\frac{dQG_1}{\delta d_1d_2}} H\left(\frac{ndDG_1G_2\beta_PQ^2}{2xy\delta d_2d_4} \right)dxdy+\BigO{(\tau(d_2)+\tau(d_4))Q\log^2Q}.\numberthis\label{l20}
 \end{align*}
We denote the double integral in the above estimate by $\mathcal{I}$. By suitable change of variables, we obtain
 \begin{align*}
   \mathcal{I}=&\frac{Q^2}{d_2d_4}dDG_1G_2 \int_0^{\frac{1}{\delta d_3}} \int_0^{\frac{1}{\delta d_1}}H\left(\frac{n\beta_P}{2\delta xy} \right)dxdy.
 \end{align*}
 We further put $\frac{n\beta_P}{2\delta xy}=\lambda$ to obtain
 \begin{align*}
   \mathcal{I}&=\frac{Q^2n\beta_P}{2\delta d_2d_4}dDG_1G_2  \int_{\frac{n\delta d_1d_3\beta_P}{2}}^{\Lambda}\int_{\frac{nd_3\beta_P}{2\lambda}}^{\frac{1}{\delta d_1}}\frac{H(\lambda)}{\lambda^2x}dxd\lambda \\&
   =\frac{Q^2n\beta_P}{2\delta d_2d_4}dDG_1G_2\int_{\frac{n\delta d_1d_3\beta_P}{2}}^{\Lambda}\frac{H(\lambda)}{\lambda^2}\log\left(\frac{2\lambda}{n\delta d_1d_3\beta_P} \right)d\lambda. 
 \end{align*}
The above estimate in conjunction with \eqref{l3} and \eqref{l20} gives
 \begin{align*}
   S_{Q,P}=&\frac{Q^2\beta_P}{2\zeta(2)}\sum_{\substack{d_2,d_4\leq Q}}\frac{\mu(d_2)\mu(d_4)}{d_2d_4\phi(d_2)\phi(d_4)}\sum_{\substack{\delta, n,d_1,d_3\geq 1\\\delta nd_1d_3<\frac{2\Lambda}{\beta_P}\\\gcd\left(\frac{d_2}{dG_1},\frac{d_4}{DG_2}\right)=1}}n\mu(d_1)\mu(d_3)
    dDG_1G_2\prod_{p|\frac{d_2}{dG_1}}\left(1-\frac{1}{p}\right)\\&\times\prod_{p|\frac{d_4}{DG_2}}\left(1-\frac{1}{p}\right)\prod_{p|\frac{d_2d_4}{dDG_1G_2}}\left(1-\frac{1}{p^2}\right)^{-1} \int_{\frac{n\delta d_1d_3\beta_P}{2}}^{\Lambda}\frac{H(\lambda)}{\lambda^2}\log\left(\frac{2\lambda}{n\delta d_1d_3\beta_P} \right)d\lambda+\BigO{Q^{\frac{3}{2}+\epsilon}}\\
   =&\frac{Q^2\beta_P}{2\zeta(2)}\sum_{d_2,d_4=1}^{\infty}\frac{\mu(d_2)\mu(d_4)}{d_2d_4\phi(d_2)\phi(d_4)}\int_0^{\Lambda}\frac{H(\lambda)}{\lambda^2}\sum_{1\leq m<\frac{2\lambda}{\beta_P}}\mathfrak{h}(m)\log\left(\frac{2\lambda}{m\beta_P}\right)d\lambda+\BigO{Q^{\frac{3}{2}+\epsilon}}\\
   =&\frac{Q^2\beta_P}{2}\int_0^{\Lambda}\mathfrak{g}_2(\lambda)H(\lambda)d\lambda+\BigO{Q^{\frac{3}{2}+\epsilon}},
 \end{align*}
 where $\mathfrak{h}(m)$ and $\mathfrak{g}_2(\lambda)$ are as defined in the statement of Theorem \ref{P(x)=x(x+1)}. The above estimate with \eqref{N_Q} yields
 \[\frac{S_{Q,P}}{\mathcal{N}_{Q,P}}=\int_0^{\Lambda}H(\lambda)\mathfrak{g}_2(\lambda)d\lambda+\BigO{Q^{-\frac{1}{2}+\epsilon}}. \]
  Using an appropriate approximation argument, we obtain 
 \[\lim_{Q\to\infty}\mathcal{S}_{\mathcal{F}_{Q,P}}(\Lambda)=\int_0^{\Lambda}\mathfrak{g}_2(\lambda)d\lambda. \]
 This completes the proof of Theorem \ref{P(x)=x(x+1)}.
 Using similar arguments as in the above proof, one can also establish the pair correlation function for the sequence $\left(\mathcal{F}_{Q,P}\right)_{Q}$ for polynomials $P(x)=x(x+c)$ with non-zero discriminant.
    \section{Polynomial Farey fractions with prime denominators}
\subsection{Proof of Theorem \ref{thm3}}
In order to prove Theorem \ref{thm3}, we intent to estimate, for any positive real number $\Lambda$, subset of primes $\mathcal{B}_Q$, and polynomial $P(x)=c_{\nu}x^{\nu}+c_{\nu-1}x^{\nu-1}+\cdots+c_1x$, the quantity
\[\mathcal{S}_{\mathscr{M}_{\mathcal{B}_Q,P}}(\Lambda)=\frac{1}{\mathscr{N}_{\mathcal{B}_Q,P}}\#\{(\gamma_1,\gamma_2)\in \mathscr{M}_{\mathcal{B}_Q,P}^2: \gamma_1\ne\gamma_2, \gamma_1-\gamma_2\in\frac{1}{\mathscr{N}_{\mathcal{B}_Q,P}}(0,\Lambda)+\mathbb{Z}\},\numberthis\label{P1}\]
as $Q\to \infty.$
Let $H$ be any continuously differentiable function with Supp\ $H\subset(0,\Lambda)$. Define $f(y)=\sum_{n\in\mathbb{Z}}H(\mathscr{N}_{\mathcal{B}_Q,P}(y+n)),\ y\in\mathbb{R},$ and $\mathcal{S}=\sum_{\substack{\gamma_1,\gamma_2\in \mathscr{M}_{\mathcal{B}_Q,P}\\\gamma_1\ne\gamma_2 }}f(\gamma_1-\gamma_2)$. Thus, we change the problem of estimating \eqref{P1} into the exponential sum over $\mathscr{M}_{\mathcal{B}_Q,P}$. We obtain
\[\mathcal{S}=\sum_{r\in \mathbb{Z}}c_r\left|\sum_{\gamma\in \mathscr{M}_{\mathcal{B}_Q,P}}e(r\gamma)\right|^2.\numberthis\label{70}\]
Next, we estimate the exponential sum over polynomial Farey fractions with denominators in $\mathcal{B}_Q$. We have
\begin{align*}
\sum_{\gamma\in\mathscr{M}_{\mathcal{B}_Q,P}}e(r\gamma)&=\sum_{p\in\mathcal{B}_Q }\sum_{\substack{a\leq p\\\gcd(P(a),p)=1}}e\left(\frac{ar}{p}\right)
        =\sum_{\substack{p\in\mathcal{B}_Q}}\sum_{a\leq p}e\left(\frac{ar}{p}\right)\sum_{\substack{d|P(a)\\ d|p}}\mu(d)\\
        &=\sum_{\substack{p\in\mathcal{B}_Q}}\sum_{a\leq p}e\left(\frac{ar}{p}\right)-\sum_{\substack{p\in\mathcal{B}_Q}}\sum_{\substack{a\leq p\\ P(a)\equiv 0\pmod{p}}}e\left(\frac{ar}{p}\right)\\
        &=\sum_{\substack{p\in\mathcal{B}_Q\\ p|r}}p-\sum_{\substack{p\in\mathcal{B}_Q}}\sum_{\substack{a\leq p\\ P(a)\equiv 0\pmod{p}}}e\left(\frac{ar}{p}\right).
        \numberthis\label{71}
    \end{align*}
Using \eqref{70} and \eqref{71}, we obtain 
\begin{align*}
\mathcal{S}=&\sum_{r\in\mathbb{Z}}c_r\sum_{\substack{p_1\in\mathcal{B}_Q\\  p_1|r}}p_1\sum_{\substack{p_2\in\mathcal{B}_Q\\ p_2|r}}p_2-\sum_{r\in\mathbb{Z}}c_r\sum_{\substack{p_1\in\mathcal{B}_Q\\ p_1|r}}p_1\sum_{\substack{p_2\in\mathcal{B}_Q}}\sum_{\substack{a\leq p_2\\ P(a)\equiv 0\pmod{p_2}}}e\left(-\frac{ar}{p_2}\right)\\ &-\sum_{r\in\mathbb{Z}}c_r\sum_{\substack{p_1\in\mathcal{B}_Q\\ p_1|r}}p_1\sum_{\substack{p_2\in\mathcal{B}_Q}}\sum_{\substack{a\leq p_2\\ P(a)\equiv 0\pmod{p_2}}}e\left(\frac{ar}{p_2}\right)+\sum_{r\in\mathbb{Z}}c_r\left|\sum_{\substack{p\in\mathcal{B}_Q}}\sum_{\substack{a\leq p\\ P(a)\equiv 0\pmod{p}}}e\left(\frac{ar}{p}\right)\right|^2\\
    =&\sum_{\substack{p_1, p_2\in\mathcal{B}_Q}}p_1p_2\sum_{r\in\mathbb{Z}}c_{[p_1,p_2]r}-\sum_{\substack{p_1\in\mathcal{B}_Q}}p_1\sum_{\substack{p_2\in\mathcal{B}_Q}}\sum_{\substack{a\leq p_2\\ P(a)\equiv 0\pmod{p_2}}}\sum_{r\in\mathbb{Z}}c_{p_1r}e\left(-\frac{ap_1r}{p_2}\right)\\
    &-\sum_{\substack{p_1\in\mathcal{B}_Q}}p_1\sum_{\substack{p_2\in\mathcal{B}_Q}}\sum_{\substack{a\leq p_2\\ P(a)\equiv 0\pmod{p_2}}}\sum_{r\in\mathbb{Z}}c_{p_1r}e\left(\frac{ap_1r}{p_2}\right)+\sum_{\substack{p_1,p_2\in\mathcal{B}_Q}}\sum_{\substack{a_1\leq p_1\\a_2\leq p_2\\ P(a_1)\equiv 0\pmod{p_1}\\P(a_2)\equiv 0\pmod{p_2}}}\sum_{r\in\mathbb{Z}}c_re\left(\frac{a_1r}{p_1}-\frac{a_2r}{p_2}\right).\numberthis\label{prime1}
\end{align*}
 Since $c_r$ is a Fourier coefficient, we estimate the sum of Fourier coefficients in the first term on the right-hand side of the above identity using Proposition \ref{Poisson} as in \eqref{l14}. Therefore, 
\[\sum_{r\in\mathbb{Z}}c_{[p_1,p_2]r}
    =\sum_{r\in\mathbb{Z}}\frac{1}{[p_1,p_2]}H\left(\frac{r\mathscr{N}_{\mathcal{B}_Q,P}}{[p_1,p_2]} \right), \]
  The other sums of Fourier coefficients are estimated using Proposition \ref{2.8} as demonstrated in \eqref{l15} 
    \[\sum_{r\in\mathbb{Z}}c_{p_1r}e\left(\pm\frac{ap_1r}{p_2}\right)=\sum_{r\in\mathbb{Z}}\frac{1}{p_1}H\left(\frac{\mathscr{N}_{\mathcal{B}_Q,P}}{p_1}\left(n\pm\frac{ap_1}{p_2}\right) \right)  \]
    and
    \[\sum_{r\in\mathbb{Z}}c_{r}e\left(\frac{a_1r}{p_1}-\frac{a_2r}{p_2}\right)=\sum_{r\in\mathbb{Z}}H\left(\mathscr{N}_{\mathcal{B}_Q,P}\left(r+\frac{a_1}{p_1}-\frac{a_2}{p_2}\right) \right). \]
 Inserting the above estimates into \eqref{prime1} gives
 \begin{align*}
 \mathcal{S}=&\sum_{\substack{p_1, p_2\in\mathcal{B}_Q}}\frac{p_1p_2}{[p_1,p_2]}\sum_{r\in\mathbb{Z}}H\left(\frac{r\mathscr{N}_{\mathcal{B}_Q,P}}{[p_1,p_2]} \right)-\sum_{\substack{p_1\in\mathcal{B}_Q}}\sum_{\substack{p_2\in\mathcal{B}_Q}}\sum_{\substack{a\leq p_2\\ P(a)\equiv 0\pmod{p_2}}}\sum_{r\in\mathbb{Z}}\left(H\left(\frac{\mathscr{N}_{\mathcal{B}_Q,P}}{p_1}\left(n-\frac{ap_1}{p_2}\right) \right) \right. \\ &+  \left.H\left(\frac{\mathscr{N}_{\mathcal{B}_Q,P}}{p_1}\left(n+\frac{ap_1}{p_2}\right) \right)\right)+\sum_{\substack{p_1,p_2\in\mathcal{B}_Q}}\sum_{\substack{a_1\leq p_1\\a_2\leq p_2\\ P(a_1)\equiv 0\pmod{p_1}\\P(a_2)\equiv 0\pmod{p_2}}}\sum_{r\in\mathbb{Z}}H\left(\mathscr{N}_{\mathcal{B}_Q,P}\left(r+\frac{a_1}{p_1}-\frac{a_2}{p_2}\right) \right). \numberthis\label{prime2} 
 \end{align*}
    Since Supp $H\subset(0,\Lambda)$ for some $\Lambda>0$. Thus, for sufficiently large $Q$, we have  
    $H\left(\mathscr{N}_{\mathcal{B}_Q,P}\left(r+\frac{a_1}{p_1}-\frac{a_2}{p_2}\right) \right)=0$. 
 If $\sum_{p\in\mathcal{B}_Q}p^2=o\left((\mathscr{N}_{\mathcal{B}_Q,P})^2\right)$, then $p^2=o\left((\mathscr{N}_{\mathcal{B}_Q,P})^2\right)$ for all $p\in\mathcal{B}_Q$, this yields $H\left(\frac{\mathscr{N}_{\mathcal{B}_Q,P}}{p_1}\left(n\pm\frac{ap_1}{p_2}\right) \right)=0$ for large $Q$. Hence
    \[\mathcal{S}=\sum_{\substack{p_1, p_2\in\mathcal{B}_Q }}\frac{p_1p_2}{[p_1,p_2]}\sum_{r\in\mathbb{Z}}H\left(\frac{r\mathscr{N}_{\mathcal{B}_Q,P}}{[p_1,p_2]} \right)=\sum_{\substack{p_1, p_2\in\mathcal{B}_Q\\ p_1\ne p_2\\p_1p_2>\mathscr{N}_{\mathcal{B}_Q,P}/\Lambda}}\sum_{r\in\mathbb{Z}}H\left(\frac{r\mathscr{N}_{\mathcal{B}_Q,P}}{p_1p_2} \right). \]
 In the last step, we used the fact that Supp $H\subset(0,\Lambda)$. We employ \cite[Lemma 3]{Xiong} to estimate the inner sum and use the fact that derivative of a smooth, compactly supported function is compactly supported to obtain 
\begin{align*}
    \mathcal{S}&=\frac{\int_{\mathbb{R}}H(x)dx}{\mathscr{N}_{\mathcal{B}_Q,P}}\sum_{\substack{p_1, p_2\in\mathcal{B}_Q\\ p_1\ne p_2}}p_1p_2+\BigOH{\sum_{p_1, p_2\in\mathcal{B}_Q}1}=\frac{\int_{\mathbb{R}}H(x)dx}{\mathscr{N}_{\mathcal{B}_Q,P}}\left(\sum_{\substack{p_1, p_2\in\mathcal{B}_Q}}p_1p_2-\sum_{p\in\mathcal{B}_Q}p^2 \right)+\BigOH{(\#\mathcal{B}_Q)^2}\\
&=\mathscr{N}_{\mathcal{B}_Q,P}(1+o(1))\int_{\mathbb{R}}H(x)dx+\BigOH{(\#\mathcal{B}_Q)^2}.
\end{align*}
Therefore,
\[\frac{\mathcal{S}}{\mathscr{N}_{\mathcal{B}_Q, P}}=(1+o(1))\int_{\mathbb{R}}H(x)dx+\BigOH{\frac{(\#\mathcal{B}_Q)^2}{\mathscr{N}_{\mathcal{B}_Q, P}}}. \numberthis\label{S/N}\]
Suppose $\#\mathcal{B}_Q=m$. Then, by Proposition \ref{M_BQ}, we have
\begin{align*}
    \mathscr{N}_{\mathcal{B}_Q, P}\gg \sum_{k=1}^mp_k\sim \sum_{k=1}^mk\log k\gg m^2\log m,
\end{align*}
where $p_k$ is the $k$th prime. Thus ${(\#\mathcal{B}_Q)^2}=o{(\mathscr{N}_{\mathcal{B}_Q, P})}$. By \eqref{S/N}, we obtain
\[\lim_{Q\to\infty}\frac{\mathcal{S}}{\mathscr{N}_{\mathcal{B}_Q, P}}=\int_{\mathbb{R}}H(x)dx. \]
Using the standard approximation argument, we approximate the characteristic function of the interval $(0,\Lambda)$ from above and below by  smooth functions with compact support in $(0,\Lambda)$, we deduce that the pair correlation function of the sequence $(\mathscr{M}_{\mathcal{B}_Q, P})_{Q\geq 1}$ is constant equal to 1 as $Q\to\infty$.
 
 For the other direction, suppose that the pair correlation of $\mathscr{M}_{\mathcal{B}_Q, P}$ is Poissonian and on the contrary $\sum_{p\in\mathcal{B}_Q}p^2$ is not $o\left((\#\mathscr{M}_{\mathcal{B}_Q,P})^2\right)$. In \eqref{prime2}, the characteristic function of the interval $(0,\Lambda)$ is approximated from both above and below by the smooth functions with compact support in $(0,\Lambda)$, we obtain
 \begin{align*}
    \mathscr{N}_{\mathcal{B}_Q,P}\mathcal{S}_{\mathscr{M}_{\mathcal{B}_Q,P}}(\Lambda)=&\sum_{\substack{p_1, p_2\in\mathcal{B}_Q\\ p_1\ne p_2}}\left[\frac{p_1p_2\Lambda}{\mathscr{N}_{\mathcal{B}_Q,P}} \right]+\sum_{\substack{p\in\mathcal{B}_Q}}p\left[\frac{p\Lambda}{\mathscr{N}_{\mathcal{B}_Q,P}} \right]-\sum_{\substack{p_1,p_2\in\mathcal{B}_Q}}\sum_{\substack{a\leq p_2\\P(a)\equiv 0\pmod{p_2}}}\left[\frac{p_1\Lambda}{\mathscr{N}_{\mathcal{B}_Q,P}}+\frac{ap_1}{p_2} \right] \\ &-\sum_{\substack{p_1,p_2\in\mathcal{B}_Q}}\sum_{\substack{a\leq p_2\\P(a)\equiv 0\pmod{p_2}}}\left[\frac{p_1\Lambda}{\mathscr{N}_{\mathcal{B}_Q,P}}-\frac{ap_1}{p_2} \right]\\
    =&\frac{\Lambda}{\mathscr{N}_{\mathcal{B}_Q,P}}\sum_{\substack{p_1, p_2\in\mathcal{B}_Q}}p_1p_2-\sum_{\substack{p\in\mathcal{B}_Q}}p\left\{\frac{p\Lambda}{\mathscr{N}_{\mathcal{B}_Q,P}} \right\} -2{\sum_{p_1,p_2\in\mathcal{B}_Q}\sum_{\substack{a\leq p_2\\P(a)\equiv 0\pmod{p_2}}}\frac{p_1\Lambda}{\mathscr{N}_{\mathcal{B}_Q,P}} }\\ & +\BigO{(\#\mathcal{B}_Q)^2},
 \end{align*}
where $\{x \}$ is the fractional part of $x$. If $p=o(\mathscr{N}_{\mathcal{B}_Q,P})$ for every $p\in\mathcal{B}_Q$, then $\lim_{Q\to\infty}\frac{1}{\mathscr{N}_{\mathcal{B}_Q,P}}\sum_{\substack{p\in\mathcal{B}_Q}}p\left\{\frac{p\Lambda}{\mathscr{N}_{\mathcal{B}_Q,P}} \right\}$ $ = \lim_{Q\to\infty}\frac{\Lambda}{(\mathscr{N}_{\mathcal{B}_Q,P})^2}\sum_{\substack{p\in\mathcal{B}_Q}}p^2\ne 0$. If there exists a $p\in\mathcal{B}_Q$ such that $p$ is not $o(\mathscr{N}_{\mathcal{B}_Q,P})$, then $\lim_{Q\to\infty}\frac{1}{\mathscr{N}_{\mathcal{B}_Q,P}} $ $ \times\sum_{\substack{p\in\mathcal{B}_Q}}p\left\{\frac{p\Lambda}{\mathscr{N}_{\mathcal{B}_Q,P}} \right\}\geq\lim_{Q\to\infty}\frac{1}{\mathscr{N}_{\mathcal{B}_Q,P}}p\left\{\frac{p\Lambda}{\mathscr{N}_{\mathcal{B}_Q,P}} \right\}\ne 0$. Therefore,
$\lim_{Q\to \infty}{\mathcal{S}_{\mathscr{M}_{\mathcal{B}_Q,P}}(\Lambda)}\ne \Lambda$, thus the pair correlation is not Poissonian, which is a contradiction. This completes the proof of Theorem \ref{thm3}.

\subsection{Proof of Corollary \ref{cor2}}
Corollary \ref{cor2} concerns the distribution of polynomial Farey fractions with denominators in the set of Piatetski-Shapiro primes, the set of Chen primes and the set of Hardy-Littlewood prime k-tuples. We use the following counting estimates to verify the condition 
\[\sum_{p\in\mathcal{B}_Q}p^2=o\left((\#\mathscr{M}_{\mathcal{B}_Q,P})^2\right), \numberthis\label{PS1}\]
in Theorem \ref{thm3} for these sets. The following theorems are useful to establish \eqref{PS1}. 
\begin{thm}\cite{MR1820915}\label{PS PNT}
Let $\pi_c(x)$ denote the number of Piatetski-Shapiro primes $p\leq x$. Then for fixed $c\in (1, 1.16)$, we have
\[\pi_c(x)\sim \frac{x^{1/c}}{\log x}\ \text{as}\ x\to \infty. \]
\end{thm}
\begin{thm}\cite[Theorem II]{Chen}\label{Chen PNT}
    There exist infinitely many primes $p$ such that $p+h$ is a product of at most $2$ primes, $h$ being any even integer, and
    \[\#\{p\leq x\ |\ p+h\ \text{is a product of at most two primes} \}\geq \frac{0.67xC_h}{(\log x)^2}, \]
    where $C_h$ is a constant depending on $h$.
\end{thm}



Case-I: Let $c\in (1, 1.16)$ be a fixed real number. Suppose $\mathcal{B}_Q^{(1)}$ is the set of Piatetski-Shapiro primes less than or equal to $Q$. We use Proposition \ref{Abel} and Theorem \ref{PS PNT} to estimate the sum on the left-hand side of \eqref{PS1}:
\begin{align*}
    \sum_{p\in\mathcal{B}_Q^{(1)}}p^2&=\sum_{\substack{p\leq Q\\ p=\lfloor n^c\rfloor}}p^2=Q^2\sum_{\substack{p\leq Q\\ p=\lfloor n^c\rfloor}}1-\int_2^{Q}2t\left(\sum_{\substack{p\leq t\\ p=\lfloor n^c\rfloor}}1\right) dt\\
    &=Q^2\left(\frac{Q^{1/c}}{\log Q}+\BigO{\frac{Q^{1/c}}{\log^2Q}} \right)-2\int_2^{Q}\left(\frac{t^{1+\frac{1}{c}}}{\log t}+\BigO{\frac{t^{1+\frac{1}{c}}}{\log^2t}} \right) dt 
    =\frac{Q^{2+\frac{1}{c}}}{(2c+1)\log Q}+\BigO{\frac{Q^{2+\frac{1}{c}}}{\log^2Q}}.\numberthis\label{PS2}
\end{align*}
 To estimate the right-hand side of \eqref{PS1}, we use Propositions \ref{M_BQ}, \ref{Abel}, and Theorem \ref{PS PNT}. This yields
\begin{align*}
   \#\mathscr{M}_{\mathcal{B}_Q^{(1)},P}&=\sum_{\substack{p\in \mathcal{B}_Q^{(1)}}}p+\BigO{\#\mathcal{B}_Q^{(1)}}=\sum_{\substack{p\leq Q\\ p=\lfloor n^c\rfloor}}p+\BigO{\#\mathcal{B}_Q^{(1)}}\\ &=Q\left(\frac{Q^{1/c}}{\log Q}+\BigO{\frac{Q^{1/c}}{\log^2Q}} \right)-\int_2^{Q}\left(\frac{t^{\frac{1}{c}}}{\log t}+\BigO{\frac{t^{\frac{1}{c}}}{\log^2t}} \right) dt
=\frac{Q^{1+\frac{1}{c}}}{(c+1)\log Q}+\BigO{\frac{Q^{1+\frac{1}{c}}}{\log^2 Q}}. 
\end{align*}
Thus the above estimate with \eqref{PS2} gives the required result.

Case-II: Suppose $\mathcal{B}_Q^{(2)}$ is the set of Chen primes not exceeding $Q$. To prove the asymptotic formula in \eqref{PS1} for Chen primes, we apply Propositions \ref{M_BQ}, \ref{Abel}, and Theorem \ref{Chen PNT} to obtain
\begin{align*}
   \sum_{p\in\mathcal{B}_Q^{(2)}}p^2&=\sum_{\substack{p\leq Q\\ p\ \text{Chen prime}}}p^2=Q^2\sum_{\substack{p\leq Q\\p\ \text{Chen prime}}}1-\int_2^Q2t\left(\sum_{\substack{p\leq t\\p\ \text{Chen prime}}}1\right)dt\\ &\ll Q^2\cdot \frac{Q}{\log Q}-2\int_2^Q\frac{t^2}{(\log t)^2}dt 
   \ll \frac{Q^3}{\log Q},
\end{align*}
and
\begin{align*}
   \#\mathscr{M}_{\mathcal{B}_Q^{(2)},P}&=\sum_{\substack{p\in \mathcal{B}_Q^{(2)}}}p+\BigO{\#\mathcal{B}_Q^{(2)}}=\sum_{\substack{p\leq Q\\ p\ \text{Chen prime}}}p+\BigO{\#\mathcal{B}_Q^{(2)}}\\
   &=Q\sum_{\substack{p\leq Q\\p\ \text{Chen prime}}}1-\int_2^Q\left(\sum_{\substack{p\leq t\\p\ \text{Chen prime}}}1\right)dt +\BigO{\#\mathcal{B}_Q^{(2)}}
   \gg \frac{Q^2}{\log^2Q}.
\end{align*}

Case-III: Let $\mathcal{B}_Q^{(3)}$ be a set of prime $k$-tuples that are less than or equal to $Q$. We use Proposition \ref{M_BQ} and \ref{Abel} with the Hardy-Littlewood conjecture to derive the asymptotic formulas
\begin{align*}
  \sum_{p\in\mathcal{B}_Q^{(3)}}p^2&=\sum_{\substack{p\leq Q\\p\ \text{prime $k$-tuple}}} p^2 =Q^2\sum_{\substack{p\leq Q\\p\ \text{prime $k$-tuple}}}1-\int_2^Q2t\left(\sum_{\substack{p\leq t\\p\ \text{prime $k$-tuple}}}1\right)dt\\
  &=(1+o(1))\frac{c_kQ^3}{(\log Q)^k}-2\int_2^Q(1+o(1))\frac{c_kt^2}{(\log Q)^k}dt
  =(1+o(1))\frac{c_kQ^3}{3(\log Q)^k},
\end{align*}
and
\begin{align*}
   \#\mathscr{M}_{\mathcal{B}_Q^{(3)},P}&= \sum_{\substack{p\in \mathcal{B}_Q^{(3)}}}p+\BigO{\#\mathcal{B}_Q^{(3)}}=(1+o(1))\frac{c_kQ^2}{(\log Q)^k}-\int_2^Q(1+o(1))\frac{c_kt}{(\log Q)^k}dt
   =(1+o(1))\frac{c_kQ^2}{2(\log Q)^k},
\end{align*}
where $c_k$ is a positive constant depending on $k$. Collecting the above estimates establishes the required result in \eqref{PS1}. This completes the proof of Corollary \ref{cor2}.

\subsection{Proof of Corollary \ref{cor3}} Suppose $\mathcal{B}_Q^{(4)}$ is the set of primes that are less than or equal to $Q$ with restricted digits. The following result of Maynard is useful in establishing Corollary \ref{cor3}.
\begin{thm}\cite[Theorem 1.2]{Maynard}\label{Maynard PNT}
  Let $q$ be sufficiently large, and let $X\geq q$. For any choice $B\subset\{0,\ldots,q-1 \}$ with $\#B=s\leq q^{23/80}$, let 
  \[A=\left\{\sum_{0\leq i\leq k}n_iq^i\ :\ n_i\in\{0,1,\ldots,q-1 \}\setminus B,\ k\geq 0\right\} \]
  be the set of integers less than $X$ with no digit in base $q$ in the set $B$. Then we have
  \[\#\{p\in A \}\asymp \frac{X^{\log(q-s)/\log q}}{\log X}. \]
\end{thm}
 Now, employing Proposition \ref{Abel} and the above theorem, we have
 \begin{align*}
    \sum_{p\in\mathcal{B}_Q^{(4)}}p^2&=\sum_{\substack{p\leq Q\\p\in\mathbf{A}}}p^2=Q^2\sum_{\substack{p\leq Q\\p\in\mathbf{A} }}1-\int_2^Q2t\left(\sum_{\substack{p\leq t\\p\in\mathbf{A}}}1\right)dt\asymp Q^2\cdot\frac{Q^{\frac{\log(q-\#\mathcal{D})}{\log q}}}{\log Q}-2\int_2^Q\frac{t^{1+\frac{\log(q-\#\mathcal{D})}{\log q}}}{\log t}dt 
\asymp\frac{Q^{2+\frac{\log(q-\#\mathcal{D})}{\log q}}}{\log Q}. 
 \end{align*}
 The right-hand side of \eqref{PS1} is estimated using Propositions \ref{M_BQ}, \ref{Abel}, and Theorem \ref{Maynard PNT}
 \begin{align*}
    \#\mathscr{M}_{\mathcal{B}_Q^{(4)},P}&=\sum_{\substack{p\in \mathcal{B}_Q^{(4)}}}p+\BigO{\#\mathcal{B}_Q^{(4)}}=\sum_{\substack{p\leq Q\\p\in\mathbf{A}}}p+\BigO{\#\mathcal{B}_Q^{(4)}}=Q\sum_{\substack{p\leq Q\\p\in\mathbf{A} }}1-\int_2^Q\left(\sum_{\substack{p\leq t\\p\in\mathbf{A}}}1\right)dt \\
    &\asymp Q\cdot\frac{Q^{\frac{\log(q-\#\mathcal{D})}{\log q}}}{\log Q}-\int_2^Q\frac{t^{\frac{\log(q-\#\mathcal{D})}{\log q}}}{\log t}dt 
    \asymp\frac{Q^{1+\frac{\log(q-\#\mathcal{D})}{\log q}}}{\log Q}. 
 \end{align*}
 Therefore, the above two estimates together with \eqref{PS1} give the required result. 

\section{Races for Farey fractions}
\subsection{Proof of Theorem \ref{thm4}}
We use the approach of Knapowski and Tur\'{a}n \cite{Knapowski} in proving Theorem \ref{thm4}. We begin with the integral
\[\int_1^{\infty}(S(Q;q,l_1)-S(Q;q,l_2)\pm cQ^{\frac{1}{2}-\epsilon})Q^{-s-1}dQ,\ \text{for}\ \epsilon>0.\]
Now, the difference can be estimated using \eqref{I7} as 
\begin{align*}
    S(Q;q,l_1)-S(Q;q,l_2)&=\left(\sum_{\substack{n\leq Q\\n\equiv l_1\pmod{q}}}1-\sum_{\substack{n\leq Q\\n\equiv l_2\pmod{q}}}1\right)\sum_{\substack{a\leq n\\\gcd(P(a),n)=1}}1\\
    &=\frac{1}{\phi(q)}\left(\sum_{\chi\pmod q}(\bar{\chi}(l_1)-\bar{\chi}(l_2))\right)\sum_{n\leq Q}n\chi(n)K(n),
\end{align*}
where $K(n)=\sum_{d|n}\frac{\mu(d)f_{P}(d)}{d}$.
Therefore,
\begin{align*}
   \int_1^{\infty}(S(Q;q,l_1)-S(Q;q,l_2)\pm cQ^{\frac{1}{2}-\epsilon})Q^{-s-1}dQ=& \frac{1}{\phi(q)}\sum_{\chi\pmod q}(\bar{\chi}(l_1)-\bar{\chi}(l_2))\int_1^{\infty}\frac{\sum_{n\leq Q}n\chi(n)K(n)}{Q^{s+1}}dQ\\
   &\pm \frac{c}{s-\frac{1}{2}+\epsilon}.\numberthis\label{C1}
\end{align*}
 The integral on the right-hand side represents the Dirichlet series $\sum_{n=1}^{\infty}\frac{\chi(n)K(n)}{n^{s-1}}$ as a Mellin transform. Using \eqref{C3}, we have
\begin{align*}
  \int_1^{\infty}\frac{\sum_{n\leq Q}n\chi(n)K(n)}{Q^{s+1}}dQ &= \frac{1}{s}\sum_{n=1}^{\infty}\frac{\chi(n)K(n)}{n^{s-1}}\\&=\frac{L(s-1,\chi)}{s}\prod_{p}\left(1-\frac{\chi(p)f_{P}(p)}{p^s}\right),\ \text{for}\ \sigma>2.\numberthis\label{C2}
\end{align*}
This along with \eqref{C1} gives
\begin{align*}
&\int_1^{\infty}(S(Q;q,l_1)-S(Q;q,l_2)\pm cQ^{\frac{1}{2}-\epsilon})Q^{-s-1}dQ= \frac{1}{\phi(q)}\sum_{\chi\pmod q}(\bar{\chi}(l_1)-\bar{\chi}(l_2))\frac{1}{s}L(s-1,\chi)\\ &\times\prod_{p}\left(1-\frac{\chi(p)f_{P}(p)}{p^s}\right)
   \pm \frac{c}{s-\frac{1}{2}+\epsilon}\\&= \frac{1}{\phi(q)}\sum_{\substack{\chi\pmod q\\\chi\ne\chi_0}}(\bar{\chi}(l_1)-\bar{\chi}(l_2))\frac{1}{s}L(s-1,\chi)\prod_{p}\left(1-\frac{\chi(p)f_{P}(p)}{p^s}\right)
   \pm \frac{c}{s-\frac{1}{2}+\epsilon}.\numberthis\label{C4}
\end{align*}
In the last step, we used the fact that $\bar{\chi}_0(l_1)=\bar{\chi}_0(l_2)=1$ for $\gcd(l_1l_2,q)=1$. We will apply Proposition \ref{Landau} with
\[\mathscr{A}(Q)=S(Q;q,l_1)-S(Q;q,l_2)\pm cQ^{\frac{1}{2}-\epsilon}.\]
Thus by \eqref{C4}, we have
\[g(s)=\int_1^{\infty}\mathscr{A}(Q)Q^{-s-1}dQ=\frac{1}{s\phi(q)}\sum_{\substack{\chi \pmod q\\\chi\ne\chi_0}}(\bar{\chi}(l_1)-\bar{\chi}(l_2))L(s-1,\chi) \prod_{p}\left(1-\frac{\chi(p)f_{P}(p)}{p^s}\right)
   \pm \frac{c}{s-\frac{1}{2}+\epsilon}. \]
 Since $L(s-1,\chi)$ is entire for non-principal Dirichlet character $\chi$ (mod $q$) and $\prod_{p}\left(1-\frac{\chi(p)f_{P}(p)}{p^s}\right)$ is absolutely convergent for $\sigma>1$ so, $g(s)$ is analytic in the half plane $\sigma>1$.
Let $\Delta\ne 0$ be discriminant of $P(x)$. The product term in the last equation can be written as
\begin{align*}
     \prod_{p}\left(1-\frac{\chi(p)f_{P}(p)}{p^s}\right)&=\prod_{p|\Delta}\left(1-\frac{\chi(p)f_{P}(p)}{p^s}\right)\prod_{p\nmid\Delta}\left(1-\frac{\chi(p)f_{P}(p)}{p^s}\right)\numberthis\label{C5}
 \end{align*}
  Note that $f_P(p)$ is non-negative integer for all prime $p$. Therefore, by the binomial theorem, we obtain
  \begin{align*}
    \prod_{p\nmid\Delta}\left(1-\frac{\chi(p)}{p^s} \right)^{f_P(p)}&=\prod_{p\nmid\Delta}\left(1-\frac{f_P(p)\chi(p)}{p^s}+\frac{f_P(p)(f_P(p)-1)\chi(p^2)}{2p^{2s}}+\cdots+\frac{(-1)^{f_P(p)}\chi(p^{f_P(p)})}{p^{f_P(p)s}} \right)\\ &= \prod_{p\nmid\Delta}\left(1-\frac{\chi(p)f_P(p)}{p^s} \right)\mathscr{P}_1(\Delta).
  \end{align*}
where \[\mathscr{P}_1(\Delta)=\prod_{p\nmid\Delta}\left(1+\left(1-\frac{\chi(p)f_P(p)}{p^s}\right)^{-1}\left(\frac{\chi(p^2)f_P(p)(f_P(p)-1)}{2p^{2s}}+\cdots+\frac{(-1)^{f_P(p)}\chi(p^{f_P(p)})}{p^{sf_P(p)}} \right)  \right).\]
Inserting the above identity into \eqref{C5} gives
\[\prod_{p}\left(1-\frac{\chi(p)f_{P}(p)}{p^s}\right)=\prod_{p|\Delta}\left(1-\frac{\chi(p)f_{P}(p)}{p^s}\right)\prod_{p\nmid\Delta}\left(1-\frac{\chi(p)}{p^s} \right)^{f_P(p)}\mathscr{P}_1(\Delta)^{-1}\numberthis\label{CC1} \]
Let us assume that $P(x)$ is irreducible. An application of Dedekind's theorem (see Theorem 5.5.1, \cite{Murty}) implies that for primes $p\nmid\Delta$, $f_P(p)$ is the number of prime ideals of a ring of integers, $\mathcal{O}_K$, where $K=\mathbb{Q}[x]/(P(x))$. This yields
 \begin{align*}
     \prod_{p\nmid\Delta}\left(1-\frac{\chi(p)}{p^s}\right)^{f_{P}(p)}&=\prod_{p\nmid\Delta}\prod_{\substack{p\in\mathfrak{P}\subset\mathcal{O}_{K}\\\parallel\mathfrak{P}\parallel=p}}\left(1-\frac{\mathcal{\chi^{\prime}}(\mathfrak{P})}{\parallel\mathfrak{P}\parallel^s}\right)\\
&=\prod_{p\nmid\Delta}\prod_{p\in\mathfrak{P}\subset\mathcal{O}_{K}}\left(1-\frac{\chi^{\prime}(\mathfrak{P})}{\parallel\mathfrak{P}\parallel^s}\right)\prod_{p\nmid\Delta}\prod_{\substack{p\in\mathfrak{P}\subset\mathcal{O}_{K}\\\parallel\mathfrak{P}\parallel\ne p}}\left(1-\frac{\chi^{\prime}(\mathfrak{P})}{\parallel\mathfrak{P}\parallel^s}\right)^{-1}\\
&=\prod_{p}\prod_{p\in\mathfrak{P}\subset\mathcal{O}_{K}}\left(1-\frac{\chi^{\prime}(\mathfrak{P})}{\parallel\mathfrak{P}\parallel^s}\right)\prod_{p|\Delta}\prod_{p\in\mathfrak{P}\subset\mathcal{O}_{K}}\left(1-\frac{\chi^{\prime}(\mathfrak{P})}{\parallel\mathfrak{P}\parallel^s}\right)^{-1}\mathscr{P}_2(\Delta)\\
&=\frac{1}{\mathcal{L}(s,\chi^{\prime})}\prod_{p|\Delta}\left(1+\BigO{|p^{-s}|} \right)^{-1}\mathscr{P}_2(\Delta) ,\numberthis\label{C6}
 \end{align*}
where $\mathscr{P}_2(\Delta)=\prod_{p\nmid\Delta}\prod_{\substack{p\in\mathfrak{P}\subset\mathcal{O}_{K}\\\parallel\mathfrak{P}\parallel\ne p}}\left(1-\frac{\chi^{\prime}(\mathfrak{P})}{\parallel\mathfrak{P}\parallel^s}\right)^{-1}$,\ $\parallel\mathfrak{P}\parallel$ denotes the norm of a prime ideal $\mathfrak{P}$, and $\mathcal{L}$ is the Hecke $L$-function modulo $\mathfrak{q}$. So, from \eqref{CC1} and \eqref{C6}, we have
\[\prod_{p}\left(1-\frac{\chi(p)f_{P}(p)}{p^s}\right)=\frac{1}{\mathcal{L}(s,\chi^{\prime})}\prod_{p|\Delta}\left(1-\frac{\chi(p)f_{P}(p)}{p^s}\right)\prod_{p|\Delta}\left(1+\BigO{|p^{-s}|} \right)^{-1}\mathscr{P}_1(\Delta)^{-1}\mathscr{P}_2(\Delta). \]
This yields
\begin{align*}
    g(s)=&\frac{1}{s\phi(q)}\sum_{\substack{\chi \pmod q\\\chi\ne\chi_0}}(\bar{\chi}(l_1)-\bar{\chi}(l_2))\frac{L(s-1,\chi)}{\mathcal{L}(s,\chi^{\prime})}\prod_{p|\Delta}\left(1-\frac{\chi(p)f_{P}(p)}{p^s}\right)\\ &\times\prod_{p|\Delta}\left(1+\BigO{|p^{-s}|} \right)^{-1}\mathscr{P}_1(\Delta)\mathscr{P}_2(\Delta) 
   \pm \frac{c}{s-\frac{1}{2}+\epsilon}. \numberthis\label{g(s)}
\end{align*}

 Clearly ${1}/{\mathcal{L}(s,\chi^{\prime})}$ is analytic in the half plane $\sigma\geq 1$ for all character $\chi^{\prime}$ as none of the denominators $\mathcal{L}(s,\chi^{\prime})$ have zeros in the half-plane $\sigma\geq 1$. The products on the right-hand side run over the prime divisors of $\Delta\ne 0$, so they are finite. The product terms $\mathscr{P}_1(\Delta)$ and $\mathscr{P}_2(\Delta)$ are convergent in the half plane $\sigma>\frac{1}{2}$. Thus, the product $\prod_{p}\left(1-\frac{\chi(p)f_{P}(p)}{p^s}\right)$ is absolutely convergent for $\sigma\geq 1$. Hence, the function $g(s)$ is analytic in the half plane $\sigma\geq 1$. 

 By the Haselgrove's condition, $\mathcal{L}(s,\chi^\prime)\ne 0$ on the real segment $0<\sigma<1$ so, $1/\mathcal{L}(s,\chi^{\prime})$ defines an analytic function on the real segment $0<\sigma<1$.
 Hence, $g(s)$ is regular on the real segment $\frac{1}{2}-\epsilon<\sigma< 1$.

 Suppose there exists a positive constant $Q_0$ such that $\mathscr{A}(Q)$ does not change its sign for $Q>Q_0$. Then by Proposition \ref{Landau}, $g(s)$ is an analytic function in the half plane $\Re(s)>\frac{1}{2}-\epsilon$. Therefore, each zero of the denominator $\mathcal{L}(s,\chi^{\prime})$ must also be a zero of the numerator $L(s-1,\chi)$. However, it is known that all nontrivial zeros of $\mathcal{L}(s,\chi^{\prime})$ are in the critical strip $0<\sigma<1$. Similarly, all nontrivial zeros of $L(s-1,\chi)$ are in the critical strip $1<\sigma<2$. Thus, the nontrivial zeros of the denominator $\mathcal{L}(s,\chi^{\prime})$ in \eqref{g(s)} cannot be canceled by the zeros of numerator $L(s-1,\chi)$.
Since the zeros of $\mathcal{L}(s,\chi^{\prime})$ has a symmetry along critical line $\sigma=1/2$ (see \cite[Section 5]{MR447191}). Hence, there exist zeros of the denominator $\mathcal{L}(s,\chi^{\prime})$ in \eqref{g(s)} in the strip $1/2\leq\sigma<1$. Any such zero is a pole of $g(s)$ and this by Proposition \ref{Landau}, contradicts the assumption that there exists $Q_0$ such that $\mathscr{A}(Q)$ does not change sign for $Q>Q_0$ and thus we conclude the result for irreducible polynomials. We next consider that $P(x)$ is reducible with the factorization
 \[P(x)=\prod_{i=1}^Jm_i(x)^{e_i},\]
 where $m_i(x)\in\mathbb{Z}[x]$ are irreducible polynomials. Let $\Delta=\prod_{i=1}^J\Delta_i$ be the discriminant of $P(x)$, where $\Delta_i\ne 0$ is discriminant of $m_i(x)$. For primes $p\nmid\Delta$,
 \[f_P(p)=\sum_{i=1}^Jf_{m_i}(p),\]
where $f_{m_i}(p)$ is the number of prime ideals of a ring of integers, $\mathcal{O}_{K_i}$ and $K_i=\mathbb{Q}[x]/(m_i(x))$, by an application of Dedekind's Theorem (see Theorem 5.5.1, \cite{Murty}). We follow the proof of irreducible case and the above identity to obtain
\[\prod_{p\nmid\Delta}\left(1-\frac{\chi(p)}{p^s}\right)^{f_{P}(p)}=\prod_{i=1}^J\mathcal{L}_{i}(s,\chi^{\prime})^{-1}\prod_{p|\Delta}\left(1+\BigO{|p^{-s}|} \right)^{-1}\mathscr{P}_{2}(\Delta)^J, \]
and
\[\prod_{p}\left(1-\frac{\chi(p)f_{P}(p)}{p^s}\right)=\prod_{i=1}^J\mathcal{L}_{i}(s,\chi^{\prime})^{-1}\mathscr{P}_{1,i}(\Delta)\prod_{p|\Delta}\left(1-\frac{\chi(p)f_{P}(p)}{p^s}\right)\prod_{p|\Delta}\left(1+\BigO{|p^{-s}|} \right)^{-1}\mathscr{P}_{2}(\Delta)^J,\numberthis\label{C8} \]
where \[\mathscr{P}_{1,i}(\Delta)=\prod_{p\nmid\Delta}\left(1+\left(1-\frac{\chi(p)f_{m_i}(p)}{p^s}\right)^{-1}\left(\frac{\chi(p^2)f_{m_i}(p)(f_{m_i}(p)-1)}{2p^{2s}}+\cdots+\frac{(-1)^{f_{m_i}(p)}\chi(p^{f_{m_i}(p)})}{p^{sf_{m_i}(p)}} \right)  \right).\]
By the same argument made for irreducible case, $g(s)$ is analytic in the half-plane $\Re(s)\geq 1$ and can be analytically continued on the real segment $\frac{1}{2}-\epsilon<\sigma< 1$. Suppose there exists a positive constant $Q_0$ such that $\mathscr{A}(Q)$ does not change its sign for $Q>Q_0$. In a similar manner as for the irreducible case in view of Proposition \ref{Landau}, we get a contradiction on the assumption that $\mathscr{A}(Q)$ does not change the sign for $Q>Q_0$. This completes the proof of Theorem \ref{thm4}.

\subsection{Proof of Theorem \ref{thm5}}
In order to prove Theorem \ref{thm5}, we use Proposition \ref{omega}, which is an analog of Landau's result \cite[Theorem 1.7]{Vaughan}. We denote
\[\mathcal{A}(Q)=-S(Q;q,l)+\frac{Q^2}{2\phi(q)}\prod_{p|q}\left(1-\frac{1}{p} \right)\prod_{p\nmid q}\left(1-\frac{f_P(p)}{p^2} \right)+Q^{\Theta-\epsilon}, \]
where $\Theta$ denotes the supremum of the real parts of zeros of the Hecke $L$-functions $\mathcal{L}_{i}(s,\chi^\prime)$ modulo $\mathfrak{q}_i$. The second term on the right-hand side of the above identity is the main term in the asymptotic formula of $S(Q;q,l)$ (see Proposition \ref{S(Q,q)}). For $\sigma>2$, we consider
\begin{align*}
    \int_1^{\infty}\mathcal{A}(Q)Q^{-s-1}dQ=&\int_1^{\infty}\left(-S(Q;q,l)+\frac{Q^2}{2\phi(q)}\prod_{p|q}\left(1-\frac{1}{p} \right)\prod_{p\nmid q}\left(1-\frac{f_P(p)}{p^2} \right)+Q^{\Theta-\epsilon}\right)Q^{-s-1}dQ\\
    =&\frac{-1}{\phi(q)}\sum_{\chi\pmod q}\bar{\chi}(l)\frac{1}{s}L(s-1,\chi) \prod_{p}\left(1-\frac{\chi(p)f_{P}(p)}{p^s}\right)+\frac{1}{2\phi(q)}\prod_{p|q}\left(1-\frac{1}{p} \right)\\ &\times\prod_{p\nmid q}\left(1-\frac{f_P(p)}{p^2} \right)\frac{1}{s-2}+\frac{1}{s-\Theta+\epsilon}\\
    =&\frac{-1}{s\phi(q)}\sum_{\substack{\chi\pmod q\\\chi\ne\chi_0}}\bar{\chi}(l)L(s-1,\chi) \prod_{p}\left(1-\frac{\chi(p)f_{P}(p)}{p^s}\right)-\frac{\zeta(s-1)}{s\phi(q)}\prod_{p|q}\left(1-\frac{1}{p^{s-1}} \right)\\&\times\prod_{p}\left(1-\frac{\chi_0(p)f_{P}(p)}{p^s}\right)+\frac{1}{2\phi(q)}\prod_{p|q}\left(1-\frac{1}{p} \right) \prod_{p\nmid q}\left(1-\frac{f_P(p)}{p^2} \right)\frac{1}{s-2}+\frac{1}{s-\Theta+\epsilon}.\numberthis\label{A(Q)}
    \end{align*}
    By \eqref{C8}, we have
    \[\prod_{p}\left(1-\frac{\chi(p)f_{P}(p)}{p^s}\right)=\prod_{i=1}^J\mathcal{L}_{i}(s,\chi^\prime)^{-1}\mathscr{P}_{1,i}(\Delta)\prod_{p|\Delta}\left(1-\frac{\chi(p)f_{P}(p)}{p^s}\right)\prod_{p|\Delta}\left(1+\BigO{|p^{-s}|} \right)^{-1}\mathscr{P}_{2}(\Delta)^J. \]
Note that the second term in equation \eqref{A(Q)} has a pole at $s=2$ due to $\zeta(s-1)$, which will be canceled by the pole of the third term at $s=2$.
Thus, the right-hand side of \eqref{A(Q)} is analytic in the half-plane $\Re(s)>1$ and using the Haselgrove's condition for Hecke $L$-function, it can be analytically continued on the real segment $(\Theta-\epsilon, 1]$ with a simple pole at $s=\Theta-\epsilon$. This yields, $\int_1^{\infty}\mathcal{A}(Q)Q^{-\sigma-1}dQ<\infty$ for $\sigma>\Theta-\epsilon$. Let us assume that
\[S(Q;q,l)-\frac{Q^2}{2\phi(q)}\prod_{p|q}\left(1-\frac{1}{p} \right)\prod_{p\nmid q}\left(1-\frac{f_P(p)}{p^2} \right)<Q^{\Theta-\epsilon}\ \text{for all}\ Q>Q_0(\epsilon). \numberthis\label{omega1}\]
Employing Proposition \ref{omega}, we deduce that the integral $\int_1^{\infty}\mathcal{A}(Q)Q^{-s-1}dQ$ is analytic in the half plane $\Re(s)>\Theta-\epsilon$. In view of the definition of $\Theta$, the product $\prod_{p}\left(1-\frac{\chi(p)f_{P}(p)}{p^s}\right)$ has poles with $\Re(s)>\Theta-\epsilon$, due to the zeros of $\mathcal{L}_{i}(s,\chi^\prime)$. This leads to a contradiction, and one can deduce that the assertion \eqref{omega1} is not true. Hence
\[S(Q;q,l)-\frac{Q^2}{2\phi(q)}\prod_{p|q}\left(1-\frac{1}{p} \right)\prod_{p\nmid q}\left(1-\frac{f_P(p)}{p^2} \right)=\Omega_{+}(Q^{\Theta-\epsilon}). \]
To derive the corresponding estimate for $\Omega_-$, we proceed in a similar manner with
\[\mathcal{A}(Q)=S(Q;q,l)-\frac{Q^2}{2\phi(q)}\prod_{p|q}\left(1-\frac{1}{p} \right)\prod_{p\nmid q}\left(1-\frac{f_P(p)}{p^2} \right)+Q^{\Theta-\epsilon}, \]
and obtain
\[S(Q;q,l)-\frac{Q^2}{2\phi(q)}\prod_{p|q}\left(1-\frac{1}{p} \right)\prod_{p\nmid q}\left(1-\frac{f_P(p)}{p^2} \right)=\Omega_{-}(Q^{\Theta-\epsilon}). \]
This completes the proof of Theorem \ref{thm5}.

\section{Appendix}
Initially, while dealing with the exponential sum over polynomial Farey fractions to obtain a closed-form formula, we ended up with an exponential sum twisted by M\"{o}bius function with coprimality condition. With this motivation, we prove a result for exponential sum twisted by the M\"{o}bius function with some coprimality restriction (a slight variation of a consequence of Dartyge and Martin.
\cite[Lemma 1]{Martin}).
\begin{lem}\label{mu exp}
        Let $q\geq 2, t\in\mathbb{Z}$, then for every $\epsilon>0$ under the GRH, we have 
        \[\sum_{\substack{n\leq z\\\gcd(n,q)=1}}\mu(n)e\left(\frac{t\bar{n}_q}{q}\right)\ll z^{\frac{3}{4}+\epsilon}q^{\frac{1}{2}}\tau(q)^2, \numberthis\label{mu GRH}\]
        where $\tau(q)$ denotes the number of positive divisors of $q$ and $\bar{n}_q$ is an integer such that $\bar{n}_qn\equiv 1\pmod{q}$.
    \end{lem}

\begin{proof}
    We begin by collecting the sum on the left hand side of \eqref{mu GRH}
    \begin{align*}
        \sum_{\substack{n\leq z\\\gcd(n,q)=1}}\mu(n)e\left(\frac{t\bar{n}_q}{q}\right)=&\sum_{a\pmod q}\sum_{\substack{n\leq z\\\gcd(n,q)=1\\n\equiv a\pmod q}}\mu(n)e\left(\frac{t\bar{n}_q}{q}\right)
        =\sum_{\substack{a\pmod q\\\gcd(a,q)=1}}e\left(\frac{t\bar{a}}{q}\right)\sum_{\substack{n\leq z\\n\equiv a\pmod q}}\mu(n)\\
        =&\sum_{\substack{a\pmod q\\\gcd(a,q)=1}}e\left(\frac{t\bar{a}}{q}\right)\sum_{n\leq z}\mu(n)\frac{1}{q}\sum_{h=0}^{q-1}e\left(\frac{h(n-a)}{q}\right)\\
        =&\frac{1}{q}\sum_{\substack{a\pmod q\\\gcd(a,q)=1}}e\left(\frac{t\bar{a}}{q}\right)\sum_{n\leq z}\mu(n)+\frac{1}{q}\sum_{h=1}^{q-1}\left(\sum_{\substack{a\pmod q\\\gcd(a,q)=1}}e\left(\frac{t\bar{a}-ha}{q}\right) \right)\\&\times\left(\sum_{n\leq z}\mu(n)e\left(\frac{hn}{q}\right)\right)
        =S_1+S_2. \numberthis\label{Mu,e}
    \end{align*}
    Now, assuming GRH, we estimate the first sum on the right-hand side of \eqref{Mu,e} as 
    \begin{align*}
        S_1&=\frac{1}{q}\sum_{\substack{a\pmod q\\\gcd(a,q)=1}}e\left(\frac{t\bar{a}}{q}\right)\sum_{n\leq z}\mu(n)
        \ll \frac{1}{q}\sum_{a\pmod q}z^{\frac{1}{2}+\epsilon}
        \ll z^{\frac{1}{2}+\epsilon}.\numberthis\label{h=0}
    \end{align*}
Next, the second sum on the right-hand side of \eqref{Mu,e} is a complete Kloosterman sum, using the result of Hooley \cite[Lemma 2]{Hooley} and for the inner most sum using Baker and Harman\cite[Theorem]{Baker}, we have
\begin{align*}
    S_2&\ll \frac{1}{q}\sum_{h=1}^{q-1}q^{\frac{1}{2}}\gcd(h,q)^{\frac{1}{2}}\tau(q)\left|\sum_{n\leq z}\mu(n)e\left(\frac{hn}{q}\right) \right|\\
    &\ll z^{\frac{3}{4}+\epsilon}q^{-\frac{1}{2}}\tau(q)\sum_{h=1}^{q-1}\gcd(h,q)^{\frac{1}{2}}.\numberthis\label{gcd}
\end{align*}
Clearly, $\gcd(h,q)=d$ if and only if $d|h, d|q$ and $\gcd\left(\frac{h}{d},\frac{q}{d}\right)=1$. Therefore, the number of $1\leq h\leq q$ with $\gcd(h,q)=d$ are $\phi(\frac{q}{d})$. Thus, the sum in \eqref{gcd} becomes
\[\sum_{h=1}^{q}\gcd(h,q)^{\frac{1}{2}}=\sum_{d|q}d^{\frac{1}{2}}\phi\left(\frac{q}{d}\right)\leq\sum_{d|q}\frac{q}{d^{\frac{1}{2}}}\leq q\tau(q).\]
The above estimate, in conjunction with \eqref{gcd}, gives
\[S_2\ll z^{\frac{3}{4}+\epsilon}q^{\frac{1}{2}}\tau(q)^2.\numberthis\label{h ne 0}\]
Hence, the result in \eqref{mu GRH} follows from \eqref{Mu,e}, \eqref{h=0} and \eqref{h ne 0}.
\end{proof}

We next prove an asymptotic result for the Mertens function twisted by the arithmetic function $f_P$. 
\begin{lem}\label{Dlemma}
    Let $\nu\geq 2$ be an integer and let $P(x)=c_{\nu}x^{\nu}+\cdots+c_1x+c_0\in\mathbb{Z}[x]$ be a polynomial with non-zero discriminant and 
    $f_P(n)=|\{1\leq d\leq n\ | P(d)\equiv 0\pmod{n} \}|$. Then for all $x\geq 1$, we have
    \[\sum_{n\leq x}\mu(n)f_P(n)\ll x\exp{\left(-c\left(\frac{\log x}{\log\log x}\right)^{1/3}\right)},\]
    where $c$ is a constant depending on the polynomial $P(x)$.
\end{lem}
\begin{proof}
    The Dirichlet series of $\mu(n)f_P(n)$ is given by
     \[F(s)=\sum_{n=1}^{\infty}\frac{\mu(n)f_P(n)}{n^s}
        =\prod_p\left(1-\frac{f_P(p)}{p^s} \right).\]
Let us assume that $P(x)$ is irreducible, and let $\Delta\ne 0$ be the discriminant of $P(x)$. We can express the Dirichlet series $F(s)$ as
\begin{align*}  
        F(s)&=\prod_{p|\Delta}\left(1-\frac{f_P(p)}{p^s} \right)\prod_{p\nmid\Delta}\left(1-\frac{f_P(p)}{p^s} \right). \numberthis\label{D1}
    \end{align*}
  Note that $f_P(p)$ is non-negative integer for all prime $p$. Therefore, by the binomial theorem, we obtain
  \begin{align*}
    \prod_{p\nmid\Delta}\left(1-\frac{1}{p^s} \right)^{f_P(p)}&=\prod_{p\nmid\Delta}\left(1-\frac{f_P(p)}{p^s}+\frac{f_P(p)(f_P(p)-1)}{2p^{2s}}+\cdots+\frac{(-1)^{f_P(p)}}{p^{f_P(p)s}} \right)\\ &= \prod_{p|\Delta}\left(1-\frac{f_P(p)}{p^s} \right)\left(1+\BigOP{|p^{-2s}|}\right).
  \end{align*}
Inserting the above identity into \eqref{D1} yields
  \[F(s)=\prod_{p|\Delta}\left(1-\frac{f_P(p)}{p^s} \right)\prod_{p\nmid\Delta}\left(1-\frac{1}{p^s} \right)^{f_P(p)}\left(1+\BigOP{|p^{-2s}|}\right)^{-1}. \]
 An application of Dedekind's theorem (see Theorem 5.5.1, \cite{Murty}) yields that for primes $p\nmid\Delta$, $f_P(p)$ is the number of prime ideals of a ring of integers, $\mathcal{O}_K$, where $K=\mathbb{Q}[x]/(P(x))$. We obtain
\begin{align*}
    \prod_{p\nmid\Delta}\left(1-\frac{1}{p^s} \right)^{f_P(p)}&=\prod_{p\nmid\Delta}\prod_{\substack{p\in\mathcal{P}\subset\mathcal{O}_K\\||\mathcal{P}||=p}}\left(1-\frac{1}{||\mathcal{P}||^s} \right)\\
&=\prod_{p\nmid\Delta}\prod_{\substack{p\in\mathcal{P}\subset\mathcal{O}_K}}\left(1-\frac{1}{||\mathcal{P}||^s} \right)\prod_{p\nmid\Delta}\prod_{\substack{p\in\mathcal{P}\subset\mathcal{O}_K\\||\mathcal{P}||\ne p}}\left(1-\frac{1}{||\mathcal{P}||^s} \right)^{-1}\\
&=\frac{1}{\zeta_K(s)}\prod_{p|\Delta}\left(1+\BigOP{|p^{-s}|} \right)^{-1}\prod_{p\nmid\Delta}\left(1+\BigOP{|p^{-2s}|} \right)^{-1},
\end{align*}
where $\mathcal{P}$ is a prime ideal, $||\mathcal{P}||$ denotes the norm of an ideal $\mathcal{P}$, and $\zeta_K(s)$ is the Dedekind zeta function of the number field $K$. Thus, the above estimate, in conjunction with \eqref{D1}, gives
\[F(s)=\frac{1}{\zeta_K(s)}\prod_{p|\Delta}\left(1-\frac{f_P(p)}{p^s} \right)\prod_{p|\Delta}\left(1+\BigOP{|p^{-s}|} \right)^{-1}\prod_{p\nmid\Delta}\left(1+\BigOP{|p^{-2s}|}\right)^{-1}, \]
which is absolutely convergent for $\Re(s)> 1$. Moreover, the product terms are absolutely convergent for $\Re(s)>1/2$. Employing Perron's formula for Dirichlet series $F(s)$ \cite[Theorem 2, p. 132]{Tenenbaum}, we obtain for some fixed $\alpha>1$
\[\sum_{n\leq x}\mu(n)f_P(n)=\frac{1}{2\pi i}\int_{\alpha-iT}^{\alpha+iT}F(s)\frac{x^s}{s}ds+R(T), \numberthis\label{D2}\]
where $R(T)\ll \frac{x^{\alpha}}{T}\sum_{n=1}^{\infty}\frac{f_P(n)}{n^{\alpha}|\log(x/n)|}$.
Let $T\geq 2$ and $\alpha=1+\epsilon$. Employing Huxley's \cite{HuxleyM} bound for $f_P(n)$, we estimate $R(T)$ in a similar way as Davenport \cite{Davenport}:
\[R(T)\ll \frac{x^{\alpha}\log x}{T}. \]
To estimate the integral in \eqref{D2}, we use the zero-free region for $\zeta_K(s)$ (see \cite[Lemma 2.4]{MR2144965}) and shift the path of integration to a rectangular contour that includes line segments connecting the points $\alpha\pm iT \ \text{and}\ \beta\pm iT$. By Cauchy's residue theorem, we obtain
\[\frac{1}{2\pi i}\int_{\alpha-iT}^{\alpha+iT}F(s)\frac{x^s}{s}ds=\frac{1}{2\pi i}\left(\int_{\alpha-iT}^{\beta-iT}+\int_{\beta-iT}^{\beta+iT}+\int_{\beta+iT}^{\alpha+iT} \right)F(s)\frac{x^s}{s}ds=I_1+I_2+I_3. \numberthis\label{D3}\]
By Lemma 2.4 of \cite{MR2144965}, $\zeta_K(s)$ has no zeros in the half-plane $\Re(s)\geq \beta$, except for at most one real zero, where $\beta=1-c/{(\log T)^{2/3}(\log\log T)^{1/3}}$; $c$ is some positive constant. Suppose there is no real zero of $\zeta_K(s)$ in the region $\Re(s)>\beta$. We estimate the integrals $I_1$ and $I_3$ using the bounds for $\zeta_K(s)$ (see \cite[Theorem 4]{Rademacher}) and obtain
\begin{align*}
    I_1&\ll \int_{\beta}^{\alpha}\frac{x^{\sigma}}{|\sigma+iT||\zeta_K(\sigma+iT)|}d\sigma\ll \frac{T^{\epsilon}}{T}\int_{\beta}^{\alpha}x^{\sigma}d\sigma
    \ll \frac{x^{\alpha}}{T^{1-\epsilon}\log x}.\numberthis\label{D4}
\end{align*}
A similar estimate holds true for $I_3$. The integral $I_2$ is estimated as
\begin{align*}
    I_2&\ll \int_{\beta-iT}^{\beta+iT}\frac{|x^{\beta+it}|}{|\beta+it||\zeta_K(\beta+it)|}dt\ll x^{\beta}\log T\int_{0}^{T}\frac{1}{\beta +t}dt
    \ll x^{\beta}(\log T)^2. \numberthis\label{D5}
\end{align*}
Inserting \eqref{D4} and \eqref{D5} into \eqref{D3} and choosing $ T=x^{\epsilon}\exp{\left(c\frac{(\log x)^{2/3}}{(\log\log x)^{1/3}} \right)}$, we obtain 
\[\frac{1}{2\pi i}\int_{\alpha-iT}^{\alpha+iT}F(s)\frac{x^s}{s}ds\ll x\exp{\left(-c\left(\frac{\log x}{\log\log x}\right)^{1/3}\right)}. \numberthis\label{D6}\]
Suppose $\zeta_K(s)$ has a real zero near one within the rectangular contour with vertices $\alpha\pm iT, \beta\pm iT$. Then, the integrand in \eqref{D3} has a simple pole in this rectangle. By Cauchy's residue theorem and the above arguments, we have
\[\frac{1}{2\pi i}\int_{\alpha-iT}^{\alpha+iT}F(s)\frac{x^s}{s}ds\ll x\exp{\left(-c\left(\frac{\log x}{\log\log x}\right)^{1/3}\right)}. \]
Hence, collecting all estimates with $ T=x^{\epsilon}\exp{\left(c\frac{(\log x)^{2/3}}{(\log\log x)^{1/3}} \right)}$ completes the proof of Lemma \ref{Dlemma} when $P(x)$ is irreducible. Next, we assume that $P(x)$ is reducible with the factorization
\[P(x)=\prod_{i=1}^{J}m_i(x)^{e_i}, \]
where $m_i(x)\in\mathbb{Z}[x]$ are irreducible polynomials.  Let $\Delta=\prod_{i=1}^J\Delta_i$ be the discriminant of $P(x)$, where $\Delta_i\ne 0$ is discriminant of $m_i(x)$. For primes $p\nmid\Delta$,
 \[f_P(p)=\sum_{i=1}^Jf_{m_i}(p),\]
where $f_{m_i}(p)$ is the number of prime ideals of a ring of integers $\mathcal{O}_{K_i}$ and $K_i=\mathbb{Q}[x]/(m_i(x))$. We follow the proof of irreducible case along with the above identity to obtain 
\begin{align*}
    \prod_{p\nmid\Delta}\left(1-\frac{1}{p^s} \right)^{f_P(p)}&=\prod_{i=1}^J\frac{1}{\zeta_{K_i}(s)}\prod_{p|\Delta}\left(1+\BigO{|p^{-s}|} \right)^{-1}\prod_{p\nmid\Delta}\left(1+\BigO{|p^{-2s}|} \right)^{-1},
\end{align*}
and
\[F(s)=\prod_{i=1}^J\frac{1}{\zeta_{K_i}(s)}\prod_{p|\Delta}\left(1-\frac{f_P(p)}{p^s} \right)\prod_{p|\Delta}\left(1+\BigO{|p^{-s}|} \right)^{-1}\prod_{p\nmid\Delta}\left(1+\BigO{|p^{-2s}|} \right)^{-1}. \]
Using Perron's formula for the Dirichlet series $F(s)$ and proceeding in a similar way as for the irreducible case completes the proof of Lemma \ref{Dlemma}.
\end{proof}

Moreover, we give a different proof for the asymptotic formula of $\#\mathcal{F}_{Q,P}$ with a better error term.
\begin{prop}\label{prop8.3}
    Let $n\geq 1$ be an integer and let $P(x)=c_nx^n+c_{n-1}x^{n-1}+\cdots+c_1x\in\mathbb{Z}[x]$ be a polynomial with non-zero discriminant. If $\mathcal{F}_{Q,P}$ is as in \eqref{F_Q}, then
    \[\mathcal{N}_{Q,P}=\#\mathcal{F}_{Q,P}=\frac{Q^2}{2}\prod_p \left(1-\frac{f_{P}(p)}{p^2}\right)+\BigO{Q^{1+\epsilon}},\numberthis\label{NQ}\]
    where $f_{P}(m)$ is as in \eqref{f_P}.
\end{prop}
\begin{proof}
We have
\begin{align*}
    \#\mathcal{F}_{Q,P}&=\sum_{q\leq Q}\sum_{\substack{a\leq q\\ \gcd(P(a),q)=1}}1=\sum_{q\leq Q}\sum_{a\leq q}\sum_{\substack{d|P(a)\\ d|q}}\mu(d)
    =\sum_{d\leq Q}\mu(d)\sum_{\substack{q\leq Q\\d|q}}\sum_{\substack{a\leq q\\ P(a)\equiv 0\pmod d}}1\\
    &=\sum_{d\leq Q}\mu(d)\sum_{\substack{q\leq Q\\d|q}}\frac{q}{d}f_{P}(d)
    =\sum_{d\leq Q}\mu(d)f_{P}(d)\left(\frac{Q^2}{2d^2}+\BigO{\frac{Q}{d}} \right)\\
    &=\frac{Q^2}{2}\sum_{d\leq Q}\frac{\mu(d)f_{P}(d)}{d^2}+\BigO{Q\sum_{d\leq Q}\frac{1}{d}f_{P}(d)}.\numberthis\label{F_QP} 
\end{align*}
Next, we deal with the error term using Proposition \ref{sum_f} and \ref{Abel}. Therefore
\begin{align*}
    \sum_{d\leq Q}\frac{1}{d}f_{P}(d)&=\frac{1}{Q}Q(\log Q)^{J-1}+\int_1^Q\frac{t(\log t)^{J-1}}{t^2}dt
    \ll (\log Q)^{J}.\numberthis\label{f/d}
\end{align*}
The above estimate with \eqref{F_QP} gives
\begin{align*}
    \#\mathcal{F}_{Q,P}&=\frac{Q^2}{2}\sum_{d\leq Q}\frac{\mu(d)f_{P}(d)}{d^2}+\BigO{Q(\log Q)^{J}}
    =\frac{Q^2}{2}\sum_{d=1}^{\infty}\frac{\mu(d)f_{P}(d)}{d^2}+\BigO{Q^2\sum_{d>Q}\frac{f_{P}(d)}{d^2}}+\BigO{Q(\log Q)^{J}}\\
    &=\frac{Q^2}{2}\prod_p{\left(1-\frac{f_{P}(p)}{p^2}\right)}+\BigO{Q^{1+\epsilon}}.
\end{align*}
This completes the proof of Proposition \ref{prop8.3}.
\end{proof}

\bibliographystyle{plain} 
\bibliography{reference}
\end{document}